\documentclass[aihp]{imsart}

\RequirePackage{amsthm,amsmath,amsfonts,amssymb}
\RequirePackage[numbers]{natbib}
\RequirePackage{graphicx}%% uncomment this for including figures
\RequirePackage{xcolor}

\startlocaldefs
\numberwithin{equation}{section}
\theoremstyle{plain}
\newtheorem{theorem}{Theorem}[section]
\newtheorem{lemma}{Lemma}[section]
\newtheorem{proposition}{Proposition}[section]
\theoremstyle{definition}
\newtheorem{definition}{Definition}[section]
\newtheorem{prp}{Property}[section]
\endlocaldefs

\begin{document}

\begin{frontmatter}

%%%%%%%%%%%%%%%%%%%%%%%%%%%%%%%%%%%%%%%%%%%%%%
%%                                          %%
%% Enter the title of your article here     %%
%%                                          %%
%%%%%%%%%%%%%%%%%%%%%%%%%%%%%%%%%%%%%%%%%%%%%%
\title{On the Asymptotic of Coupled Spherical Integrals: An Operator-Valued Extension of the Free Fourier Transform}
%\title{A sample article title with some additional note\thanksref{T1}}
\runtitle{On the Asymptotic of Coupled Spherical Integrals}
%\thankstext{T1}{A sample of additional note to the title.}

\begin{aug}
%%%%%%%%%%%%%%%%%%%%%%%%%%%%%%%%%%%%%%%%%%%%%%%
%% Additional information such as            %%
%% identifying the corresponding author must %%
%% be included in in the Acknowledgments     %%
%% section if necessary.                     %%
%% ORCID can be inserted by command:         %%
%% \orcid{0000-0000-0000-0000}               %%
%%%%%%%%%%%%%%%%%%%%%%%%%%%%%%%%%%%%%%%%%%%%%%%
\author[A]{\inits{L. P. B.}\fnms{Levi-Pascal V. G.}~\snm{Bohnacker} \orcid{0009-0004-9358-6625} \ead[label=e1]{levi.bohnacker@fau.de}}
\author[A]{\inits{R. R. M.}\fnms{Ralf R.}~\snm{M\"uller} \orcid{0000-0003-3780-9308} \ead[label=e2]{ralf.r.mueller@fau.de}}
%%%%%%%%%%%%%%%%%%%%%%%%%%%%%%%%%%%%%%%%%%%%%%
%% Addresses                                %%
%%%%%%%%%%%%%%%%%%%%%%%%%%%%%%%%%%%%%%%%%%%%%%
\address[A]{FAU Erlangen-Nuremberg, Institute for Digital Communications, Cauerstraße 7, D-91058 Erlangen, Germany\printead[presep={,\ }]{e1,e2}}

\end{aug}

\begin{abstract}
In this paper the asymptotic exponent of coupled Harish-Chandra-Itzykson-Zuber (HCIZ) integrals is studied.
We consider $L$ rank-one HCIZ integrals which are coupled through additive exponent terms, with $L>1$ finite.
We show that the normalized asymptotic concentration exponent in the rank-one case is, under some constraints, characterized by the operator-valued R-transform.
In this way the coupled HCIZ integral provides the operator-valued free analog of the classical Laplace transform.
The result can be used to recover the asymptotics of rank-one non-self-adjoint spherical integrals, if the constant matrix is R-diagonal in the sense of free probability.
Furthermore, we show that $L$ coupled HCIZ integrals of rank $M(N)=O\left(N^{1/2-\epsilon}\right)$ can be factorized into $M(N)$ rank-one $L$-fold coupled spherical integrals, and characterize the operator-valued R-transform for R-cyclic matrices on the algebra of real, diagonal $L\times L$ matrices.
The treatment of coupled spin glasses with correlated disorder is enabled by our asymptotic exponent.
\end{abstract}

\begin{abstract}[language=french]
% Dans cet article, nous étudions l’exposant asymptotique des intégrales couplées de Harish-Chandra-Itzykson-Zuber (HCIZ).
% Nous considérons $L$ intégrales HCIZ de rang un, couplées par des termes d’exposant additifs, avec $L > 1$ fini.
% Nous montrons que l’exposant de concentration asymptotique normalisé dans le cas de rang un est, sous certaines contraintes, caractérisé par la transformée R à valeurs opératoires.
% De cette manière, l’intégrale HCIZ couplée fournit l’analogue libre à valeurs opératoires de la transformée de Laplace classique.
% Ce résultat peut être utilisé pour déterminer l’asymptotique des intégrales sphériques de rang un non auto-adjointes, si la matrice constante est R-diagonale au sens de la probabilité libre.
% De plus, nous montrons que $L$ intégrales HCIZ couplées de rang $M(N)=O\left(N^{1/2-\epsilon}\right)$ peuvent être factorisées en $M(N)$ intégrales sphériques couplées $L$-fois de rang un, et nous caractérisons la transformée R à valeurs opératoires pour les matrices R-cycliques sur l’algèbre des matrices réelles diagonales $L\times L$.
% Notre exposant asymptotique permet d’étudier les verres de spin couplés présentant un désordre corrélé.
\end{abstract}

\begin{keyword}[class=MSC]
\kwdgroup[type=primary]{\kwd{60B20}
\kwd{46L54}}
\kwdgroup[type=secondary]{\kwd{60F10}}
\end{keyword}

\begin{keyword}
\kwd{spherical integral}
\kwd{R-transform}
\kwd{free probability theory}
\kwd{operator-valued}
\end{keyword}

\end{frontmatter}

%%%%%%%%%%%%%%%%%%%%%%%%%%%%%%%%%%%%%%%%%%%%%%
%%%% Main text entry area:
\section{Introduction}\label{1_intro}
The spherical integral introduced by Harish-Chandra \cite{harish-chandraDifferentialOperatorsSemisimple1957} and independently solved by Itzykson and Zuber \cite{itzyksonPlanarApproximationII1980} is defined as the matrix integral
\begin{align}
    I_N^{(\beta)}(D_N,E_N) 
    = 
    \int 
        \exp\left\{
            N
            \mathrm{tr}\left(
                D_N 
                U^\dagger
                E_N 
                U
            \right)
        \right\}
        \mathrm d m_N^{(\beta)}(U),
    \label{eq:HCIZ}
\end{align}
where $D_N$, $E_N\in\mathbb A^{N\times N}$ are self-adjoint with $\mathbb A=\mathbb R$ for $\beta=1$ and $\mathbb A=\mathbb C$ for $\beta=2$, $U$ is in the orthogonal group $\mathcal O_N$ or the unitary group $\mathcal U_N$ for $\beta=1$ and $\beta=2$, respectively, $m_N^{(\beta)}$ denotes the corresponding Haar measure, $\mathrm{tr}(X)$ is the trace of matrix $X$, and $X^\dagger$ denotes the conjugate transpose of matrix $X$.
The case $\beta=4$ is not considered in this manuscript.
Integral \eqref{eq:HCIZ} is also referred to as the HCIZ integral or the spherical integral, the asymptotics of which have been extensively studied in the mathematics and physics communities \cite{guionnetFourierViewTransform2005,collinsNewScalingItzykson2007,tanakaAsymptoticsHarishChandraItzyksonZuberIntegrals2008,marinariReplicaFieldTheory1994}.
Applications of large system analysis in research fields such as Communications Engineering \cite{bereyhiStatisticalMechanicsMAP2019,mullerVectorPrecodingWireless2008,tulinoSupportRecoverySparsely2013} and Artificial Intelligence \cite{maillardBayesoptimalLearningExtensivewidth2024,pourkamaliBayesianExtensiveRankMatrix} have further fueled the interest in the HCIZ and related integrals.
If matrix $D_N$ is rank-one with a single non-zero eigenvalue $\theta$, the following asymptotic obtained by Guionnet and Ma\"ida \cite[Theorem 2]{guionnetFourierViewTransform2005} holds
\begin{align}
    \lim_{N\uparrow\infty} 
    \frac{1}{N}
    \log 
    I_N^{(\beta)}(D_N,E_N) 
    =
    I_{\mu_E}^{\mathrm{GM}}(\theta)
    =
    \theta 
    \int_{0}^1
    \mathrm R_{\mu_E}\left(\frac{2}{\beta}\theta w\right) d w,
    \label{eq:rankOneGMasymptotic}
\end{align}
assuming that the empirical spectral measure of $E_N$ weakly converges to $\mu_E$, the spectral measure of the non-commutative random variable (RV) $E$, as $N\uparrow\infty$ with finitely bounded support and for small enough $\theta$.
Furthermore, $R_{\mu_E}(\cdot)$ is the R-transform of $\mu_E$ which is given as a power series of the free cumulants of $\mu_E$, i.e. $\kappa_n(E)$ for all $n\geq1$ \cite{mingoFreeProbabilityRandom2017}.
In classical probability theory, the cumulants of an RV $X$ are defined through the logarithm of its moment generating function $\mathbb E[e^{tX}]$, where $\mathbb E[\cdot]$ denotes expectation.
The moment generating function is essentially the Laplace/Fourier transform of the probability density function $f_X$ of $X$.
Similarly, the logarithm of the rank-one HCIZ integral returns a power series of the free cumulants of $E=\lim_{N\uparrow\infty}E_N$. 
Hence, $I_N^{(\beta)}(D_N, E_N)$ is interpreted as the free counterpart to the Laplace transform.
\\
If $D_N$ has rank $M(N)=O\left(N^{1/2-\epsilon}\right)$ for any $\epsilon>0$, \eqref{eq:HCIZ} can be asymptotically factorized into $M(N)$ rank-one HCIZ integrals as shown in \cite[Theorem 7]{guionnetFourierViewTransform2005} as
\begin{align}
\lim_{N\uparrow\infty}
\frac{1}{NM(N)}
\log I_N^{(\beta)}(D_N, E_N)
=
\lim_{N\uparrow\infty}\frac{1}{M(N)} \sum_{i=1}^{M(N)}I_{\mu_E}^{\mathrm{GM}}(\theta_i),
\label{eq:GM_result_extensive_rank}
\end{align}
where $\theta_i$ is the $i^\mathrm{th}$ non-zero eigenvalue of $D_N$.
These results are often applied to average over large random matrices, e.g. in the application of the replica method \cite{mezardSpinGlassTheory1986} for large system analysis \cite{bereyhiStatisticalMechanicsMAP2019}.
In this context, \eqref{eq:GM_result_extensive_rank} is also known as the \textit{Free Fourier Transform}.
Define $A_N=U^\dagger E_N U$, then $A_N$ is unitarily/orthogonally invariant, and the HCIZ integral can be interpreted as an averaging operation over the eigenspace of $A_N$ and the independent corresponding eigenvalue distribution.
This interpretation is extended in \cite{tanakaAsymptoticsHarishChandraItzyksonZuberIntegrals2008} to any $A_N$ (not necessarily unitarily/orthogonally invariant) under the assumption that $D_N$ is positive definite.
\\
In this paper, we analyze matrix integrals which can be expressed by a finite number $L$ of \textit{coupled} HCIZ integrals.
Specifically, we study the asymptotic exponent of the following integrals over the unitary/orthogonal group
\begin{align}
    &I_N^{(\beta)}(\{D_l\},\{E_{lk}\})
    =
    \int 
        \exp\left\{
            \sum_{l,k=1}^{L}
                N\mathrm{tr}\left(
                    D_kD_l^\dagger  U_l^\dagger E_{lk} U_k
                \right)
        \right\}
        \prod_{l=1}^{L} 
            \mathrm d m_N^{(\beta)}(U_l),
    \label{eq:L_coupled_spherical_integrals}
\end{align}
where $U_l,D_l,E_{lk}\in\mathbb A^{N\times N}$, $E_{lk}=E_{kl}^\dagger$, and for symbol $X$, $\{X_i\}$ denotes the indexed family $\{X_1,X_2,\ldots\}$.
Furthermore, all $D_k$ and $D_l$ have $M(N)$ non-zero singular values and share the right orthogonal matrix of their respective singular value decompositions.
To lighten the notation we have dropped the subscript $N$ from all matrices.
\\
The motivation to study matrix integrals of the form in \eqref{eq:L_coupled_spherical_integrals} arises from an application of the replica method to analyze Nearest Convex Hull Classification, which is introduced in \cite{nalbantovNearestConvexHull2006}.
To approximate the error rate, the covariance between two spin glasses with shared randomness is computed by treating the two thermodynamic systems in a coupled manner, as discussed by the authors in \cite{bohnackerGeometricAnalysisBlind2026}.
Averaging over the quenched, shared disorder is enabled by a modification of \eqref{eq:L_coupled_spherical_integrals}.
\\
In \cite{collinsAsymptoticsUnitaryOrthogonal2009}, matrix integrals are studied in a general framework which encompasses \eqref{eq:L_coupled_spherical_integrals}.
It is shown that the asymptotic exponent scaled by $1/N^2$ exists in the high-temperature regime, i.e. for small enough singular values of $D_l$.
Furthermore, a combinatorial model to solve such matrix integrals is developed and applied to the HCIZ integral, reproducing the results of \cite{guionnetFourierViewTransform2005}.
Applying this combinatorial framework to our result is subject to future work.
\\
Our main contributions are:
\begin{itemize}
    \item Derivation of the asymptotic exponent of \eqref{eq:L_coupled_spherical_integrals} for finite $L$ and $M(N)=1$ as a function of the operator-valued R-transform.
    \item Proof that \eqref{eq:L_coupled_spherical_integrals} factorizes into $M(N)$ rank-one integrals if $M(N)= O\left(N^{1/2-\epsilon}\right)$ for any $\epsilon>0$ as $N\uparrow\infty$.
    \item Modification of \eqref{eq:L_coupled_spherical_integrals} and corresponding exponents, such that the inner matrices $E_{lk}$ may depend on the singular value index of $D_{l}$.
    \item Characterization of the operator-valued R-transform for R-cyclic matrices on the algebra of real, diagonal matrices.
\end{itemize}
Our proofs are  based on the large deviations techniques used in \cite{guionnetFourierViewTransform2005}; we cannot, however, assume that the matrices $E_{lk}$ are jointly diagonalizable which alters the derivation of the upper and lower bounds.
Furthermore, the Haar-distributed $U_l$ are coupled through $E_{lk}$, which results in a coupling of the Gaussian measures and naturally leads us to introduce the operator-valued R-transform.
\\
In Section~\ref{2_statement_of_main_results} we present our main results, followed by some preliminaries on the operator-valued Stieltjes and R-transforms in Section~\ref{3_preliminaries}.
The corresponding proofs are given in Section~\ref{4_proofs}.

\section{Statement of the Main Results}\label{2_statement_of_main_results}
In the following we show that spherical integrals of the form \eqref{eq:L_coupled_spherical_integrals} exhibit an asymptotic exponent, similar to \eqref{eq:rankOneGMasymptotic} in an operator-valued free probability setting \cite{mingoFreeProbabilityRandom2017}.
Let us first define what we mean by convergence of empirical operator-valued moments.
\begin{definition}
    Let $\left(\mathcal A, \mathcal E, \mathcal B_L\right)$ be an operator-valued probability space as defined in \cite[Chapter 9]{speicherFreeProbabilityTheory2009}, where $\mathcal A$ is the algebra of $L\times L$ block matrices composed of $N\times N$ blocks as $N$ tends to infinity and $\mathcal B_L$ is the algebra of all $L\times L$ matrices.
    Furthermore, the expectation operation $\mathcal E:\mathcal A\to\mathcal B_L$ applies the averaged normalized trace $\varphi$ to all $L^2$ blocks of its argument.
    Define matrix $Z_N$ as
    \begin{align}
        Z_N
        =
        \begin{bmatrix}
            E_{11} & \ldots & E_{1L} \\
            \vdots & \ddots & \vdots \\
            E_{L1} & \ldots & E_{LL}
        \end{bmatrix},
    \end{align}
    with blocks $E_{lk}\in\mathbb A^{N\times N}~\forall l,k$ where $N$ is finite.
    $Z_N$ is said to converge to $Z$ if all empirical operator-valued moments of $Z_N$ converge to the corresponding operator-valued moments of $Z$, defined on the operator-valued probability space $(\mathcal A, \mathcal E, \mathcal B_L)$, as $N\uparrow\infty$.
    That is
    \begin{align}
        \lim_{N\uparrow\infty}
        \mathrm{Tr}_L  [Z_N B_1 Z_N \ldots B_{n-1} Z_N]
        =
        \mathcal E[Z B_1 Z \ldots B_{n-1} Z]
        ,
        \label{eq:operator_valued_moments}
    \end{align}
    for all $B_i\in\mathcal B_L$ and all $n\geq1$. 
    Note that we write with some abuse of notation, $Z_N B_i := Z_N (B_i\otimes \mathbb I_N)$.
    $\mathrm{Tr}_L$ is the normalized partial trace of order $L$, i.e. $\mathrm{Tr}_L[Z_N]$ returns an $L\times L$ matrix with entries $\mathrm{Tr}[E_{lk}]=\frac{1}{N}\mathrm{tr}\left(E_{lk}\right)$ at position $(l,k)$.
    \label{def:operator_valued_convergence}
\end{definition}
Let $\mathcal D_L$ be the algebra of real, diagonal, $L\times L$ matrices. 
If $Z_N$ is R-cyclic in the sense of operator-valued free probability, all moments as defined in \eqref{eq:operator_valued_moments} are diagonal for $B_i\in\mathcal{D}_L~\forall i$ \cite{mingoFreeProbabilityRandom2017}.
Therefore, convergence in the sense of Definition~\ref{def:operator_valued_convergence} implies convergence of $Z_N$ to $Z$ on the operator-valued probability space $(\mathcal A, \mathcal E_{\mathcal D}, \mathcal D_L)$ where $\mathcal E_{\mathcal D}: \mathcal A \to\mathcal D_L$.
\\
Denote the operator norm of matrix $X$ by $||X||_\infty$ which returns the maximum absolute eigenvalue if $X$ is self-adjoint, and the maximum singular value otherwise.
Furthermore, $[n]$ is the set $\{1,\ldots, n\}$, and let $\mathrm{diag}\{a_1,\ldots,a_n\}$ be the $n\times n$ diagonal matrix with on-diagonal elements $a_i~\forall i \in [n]$.
If all matrices $D_l$ have a single non-zero singular value, \cite[Theorem 2]{guionnetFourierViewTransform2005} is extended to the operator-valued case through:
\begin{theorem}[$L$-fold, coupled spherical integration for $\mathrm{rank}(D_l D_l^\dagger)=1$]
\label{thm:pL_rank1_integral}
Let $L=O(1)$, $M(N)=1$, and $D_l=O_l\Sigma_l V^\dagger$ be the singular-value decomposition of $D_l$ for all $l$, i.e. $D_l$ has a single non-zero singular value $\sqrt{\theta_l}$ at index $(1,1)$ of $\Sigma_l$, and the right unitary/orthogonal matrix $V^\dagger$ of the SVD of $D_l$ is identical for all $l$.
Define $Z_N$ as in Definition~\ref{def:operator_valued_convergence} with $E_{lk}\in\mathbb A^{N\times N}$, and $E_{lk}=E_{kl}^\dagger$.
If $Z_N$ converges to $Z$ where all moments of $Z$ are finite, $Z$ is R-cyclic, the spectral norm $||E_{lk}||_\infty<\infty$ is bounded for all $(l,k)$ as $N\uparrow\infty$, and all $\theta_l$ are small enough such that
\begin{align}
    -P=-\mathrm{diag}\left\{\frac{2\theta_1}{\beta},\ldots,\frac{2\theta_L}{\beta}\right\}
    \in \left\{\left.G^{\mathcal{D}_L}_Z(S) \right| S\in\mathcal S^Z_{\{-\}^L}\right\},
\end{align}
then
\begin{align}
    \lim_{N\uparrow\infty}
    \frac{1}{N}\log I_N^{(\beta)}(\{D_l\},\{E_{lk}\}) 
    = 
    \frac{\beta}{2}
    \int_{0}^1
        \mathrm{tr}\left(
            R_Z^{\mathcal{D}_L}(Pw)P
        \right)
        \mathrm d w
    = I_{Z}^{(\beta)}(\{\theta_l\}),
\end{align}
where $G_Z^{\mathcal D_L}$ and $R_Z^{\mathcal D_L}$ are the Stieltjes and R-transform of $Z$ on the operator-valued probability space $(\mathcal A, \mathcal E_{\mathcal D}, \mathcal D_L)$, respectively, and $\mathcal S^Z_{\{-\}^L}$ is the domain on which $G_{Z}^{\mathcal D_L}(S)$ exists on the real plane and is negative definite as defined in Properties~\ref{prop:2conditional_stieltjes_transform_on_real_plane} and \ref{prop:Lconditional_stieltjes_transform_on_real_plane}.
\end{theorem}
For $L=1$, $R_Z^{\mathcal D_1}$ is the scalar R-transform from free probability theory, and we recover $\eqref{eq:rankOneGMasymptotic}$.
For any $L>1$, $R_Z^{\mathcal{D}_L}$ can be expressed by a power series of the operator-valued free cumulants on $(\mathcal A,\mathcal E_{\mathcal D}, \mathcal D_L)$.
Hence, the $L$-fold coupled spherical integral is interpreted as the operator-valued free counterpart to the Laplace transform for R-cyclic matrices.
Furthermore, our result can be applied to recover the asymptotics of non-self-adjoint spherical integrals.
Benaych-Georges introduces the following asymptotic exponent in \cite{benaych-georgesRectangularRTransformLimit2011} for the rank-one case
\begin{align}
    I^\mathrm{BG}(\theta)
    =~&
    \lim_{N\uparrow\infty}
    \frac{1}{N}
    \log
    \int
        \exp\left\{\sqrt{N T} \theta \mathrm{Re}\left(
                \mathrm{tr}\left(
                    D_N U_N E_N V_T
                \right)
            \right)\right\}
        \mathrm d m_N^{(\beta)}(U_N)
        \mathrm d m_T^{(\beta)}(V_T)
    \label{eq:benaych_georges_rect_sph_int}
    =
    \beta \int_{0}^{\theta/\beta} \frac{C_{\mu_E}^{(\lambda)}(t^2)}{t}\mathrm dt
\end{align}
where $\mathrm{Re}:\mathbb C \to \mathbb R$ recovers the real part of  a complex number, $U_N$ and $V_T$ are $N\times N$ and $T \times T$ independent Haar-distributed orthogonal ($\beta=1$) or unitary ($\beta=2$) matrices, respectively, $E_N$ is an $N\times T$ matrix with empirical singular law $\hat \mu_E^N$ which converges to $\mu_E$ as $N\uparrow\infty$, and $D_N$ is a $T\times N$ matrix with all zero entries except for one entry of value one.
Furthermore, $\lambda=\lim_{N\uparrow\infty}N/T\in[0,1]$, and $C_{\mu_E}^{(\lambda)}$ is the \textit{rectangular R-transform with ratio $\lambda$} of the probability measure $\mu_E$.
Assume that $\lambda=1$, i.e. all matrices in \eqref{eq:benaych_georges_rect_sph_int} are square, and $\lim_{N\to\infty}E_N$ is R-diagonal.
Without loss of generality we assume that $D_N$ has its single non-zero entry at position $(1,1)$, such that
\begin{align}
    I^\mathrm{BG}(\theta)
    =
    {I_{Y}^{(\beta)}} \left(\left\{\frac{\theta}{2},\frac{\theta}{2}\right\}\right)
    =
    \frac{\beta}{2}
    \int_{0}^1
        \mathrm{tr}\left(
            R_Y^{\mathcal D_2}\left(\frac{2\theta w}{\beta}\mathbb I_2\right)\frac{2\theta}{\beta}\mathbb I_2
        \right)
    \mathrm d w,
    \label{eq:non_self-adjoint_spherical_integral}
\end{align}
with
\begin{align}
    Y_N = \begin{bmatrix}
        0 & E_N 
        \\
        E_N^\dagger & 0
    \end{bmatrix},
    \label{eq:Z_N_non_self-adjoint_spherical_integral}
\end{align}
and $\mathbb I_2$ being the $2\times 2$ identity matrix.
This equality is confirmed  by analyzing the power series $C_{\mu_E}^{(\lambda)}(t^2)=\sum_{n\geq 1} c_{2n}(E) t^{2n}$ \cite{benaych-georgesRectangularRandomMatrices2009} where $c_{2n}(E)$ is equal to entry $(1,1)$ of the operator-valued cumulant $\kappa_{2n}^{\mathcal D_2}(Y)$ as defined in \cite{mingoFreeProbabilityRandom2017} with $Y=\lim_{N\uparrow\infty} Y_N$ comprised of the off-diagonal elements $E$ and $E^\dagger$.
Note that $Y$ is R-cyclic such that Theorem~\ref{thm:pL_rank1_integral} holds, $\kappa_{2n}^{\mathcal D_2}(Y) = \mathrm{diag}\left\{\kappa_{n}(EE^\dagger),\kappa_{n}(E^\dagger E)\right\}$, and $\kappa_{2n+1}^{\mathcal D_2}(Y)=0$ for all $n\in\mathbb N_0$ \cite{nicaRCyclicFamiliesMatrices2002}.
The equality in \eqref{eq:non_self-adjoint_spherical_integral} is then obtained by expanding the power series of $\mathrm{tr}(R_Y(Pw) P)$ and $C_{\mu_E}^{(\lambda)}$ with corresponding substitution of the integration variable.
Based on this observation we postulate that Theorem~\ref{thm:pL_rank1_integral} may be extended to a rectangular setting with $U_l$ as $N_l\times N_l$ matrices.
Rigorous treatment of this case is left to future work.
\\
Theorem~\ref{thm:pL_rank1_integral} is complemented by Theorem~\ref{thm:pL_rankM_integral} to allow $\mathrm{rank}(D_lD_l^\dagger)=O\left(N^{1/2-\epsilon}\right)$ for any $\epsilon>0$, analogously to \cite[Theorem 7]{guionnetFourierViewTransform2005}:
\begin{theorem}[$L$-fold, coupled spherical integration for $\mathrm{rank}(D_l D_l^\dagger)=M(N)>1$]
\label{thm:pL_rankM_integral}
Let $L=O(1)$, $M(N)=O\left(N^{1/2-\epsilon}\right)$ for any $\epsilon>0$, and $D_l=O_l\Sigma_l V^\dagger$ be the singular-value decomposition of $D_l$ for all $l$, i.e. $D_l$ has $M(N)$ non-zero singular values $\sqrt{\theta_{l,i}}$, and the right unitary/orthogonal matrix $V^\dagger$ of the SVD of $D_l$ is identical for all $l$.
Let $Z_N$ and $Z$ be defined as in Theorem~\ref{thm:pL_rank1_integral} and all $\theta_{l,i}$ are small enough such that
\begin{align}
    -P_i=-\mathrm{diag}\left\{\frac{2\theta_{1,i}}{\beta},\ldots,\frac{2\theta_{L,i}}{\beta}\right\}
    \in \left\{\left.G^{\mathcal{D}_L}_Z(S) \right| S\in\mathcal S^Z_{\{-\}^L}\right\},
\end{align}
and $\hat \mu_D^N$, the joint empirical measure of all $\theta_{l,i}$, converges to some $\mu_D$ with analytic density
\begin{align}
\mu_D
=
\lim_{N\uparrow\infty}
\hat \mu_D^N
=
\lim_{N\uparrow\infty}
\frac{1}{M(N)}
\sum_{i=1}^{M(N)}
        \delta_{\{\theta_{l,i}\}},
\end{align}
where $\delta_{\{\theta_{l,i}\}}$ denotes the Dirac measure centered on the point $(\theta_{1,i}, \ldots, \theta_{L,i})$. 
Then, the $L$-fold coupled HCIZ integral of rank $M(N)>1$ factors into $M(N)$ coupled rank-one integrals as
\begin{align}
    \lim_{N\uparrow\infty}
    \frac{1}{N M(N)}\log I_N^{(\beta)}(\{D_l\},\{E_{lk}\}) 
    &=
    \lim_{N\uparrow\infty}
    \frac{1}{M(N)}
    \sum_{i=1}^{M(N)}
    \frac{\beta}{2}
    \int_{0}^1
        \mathrm{tr}\left(
            R_Z^{\mathcal D_L}(P_iw)P_i
        \right)
        \mathrm d w
        \\
    &=
    \int
        I^{(\beta)}_Z(\{\theta_l\})
        \mathrm d \mu_D(\{\theta_l\}).
\end{align}
\end{theorem}
Theorem~\ref{thm:pL_rankM_integral} is, however, restricted by the condition that all $\theta_{l,i}$ must have the same sign (note that a negative sign could be absorbed by $Z_N$).
This restriction is relaxed by associating every factorized rank-one HCIZ integral with an individual $Z_{N,i}$.
\begin{theorem}[Modified, $L$-fold, coupled spherical integration for rank $M(N)\geq 1$]
\label{thm:pL_rankM_integral_modification}
Let $L=O(1)$, $M(N)=O\left(N^{1/2-\epsilon}\right)$ for some $\epsilon >0$, and $D_{l,i}= O_{l,i}\Sigma_{l,i} V^\dagger$ have a single non-zero singular value $\sqrt{\theta_{l,i}}$ at the $i^\mathrm{th}$ on-diagonal position of $\Sigma_{l,i}$. 
Define the modified, $L$-fold coupled HCIZ integral by
\begin{align}
    J_N^{(\beta)}\left(\left\{D_{l,i}\right\},\left\{E_{lk,i}\right\}\right)
    =
    \int
        \exp\left\{
            \sum_{i=1}^{M(N)}
            \sum_{l,k=1}^{L}
                N\mathrm{tr}\left(
                    D_{k,i} D_{l,i}^\dagger
                    U_l^\dagger
                    E_{lk,i}
                    U_k
                \right)
        \right\}
        \prod_{l=1}^{L}
            \mathrm d m_N^{(\beta)}(U_l)
            .
\end{align}
The block matrix $Z_{N,i}$ is structured as in Definition~\ref{def:operator_valued_convergence}, i.e. the block at position $(l,k)$ of $Z_{N,i}$ is $E_{lk,i}$ with $E_{lk,i} = E_{kl,i}^\dagger\in \mathbb A^{N\times N}$ for all $(l,k)$, $Z_{N,i}$ converges to some $Z_{i}$ with finite moments as $N\uparrow\infty$, is R-cyclic, and $||E_{lk,i}||_\infty<\infty$ for all $(l,k,i)$ as $N\uparrow\infty$.
If all $\theta_{l,i}$ are sufficiently small, such that
\begin{align}
-P_{i}=-\mathrm{diag\left\{\frac{2\theta_{1,i}}\beta,\ldots,\frac{2\theta_{L,i}}\beta\right\}}
    \in \left\{\left.G_{Z_{i}}^{\mathcal D_L}(S) \right| S^{Z_i}\in\mathcal S_{\{-\}^L}\right\},
\end{align}
then
\begin{align}
    \lim_{N\uparrow\infty} \frac{1}{NM(N)}
    \log J_N^{(\beta)}\left(\left\{D_{l,i}\right\},\left\{E_{lk,i}\right\}\right)
    =
    \lim_{N\uparrow\infty}
    \frac{1}{M(N)}
    \sum_{i=1}^{M(N)}
        \frac{\beta}{2}
        \int_{0}^1
            \mathrm{tr}\left(
                R_{Z_{i}}^{\mathcal D_L}(P_i w) P_i
            \right)
            \mathrm d w.
\end{align}  
\end{theorem}\
This theorem holds for all finite $L$; the modification can hence be applied to \cite[Theorem 7]{guionnetFourierViewTransform2005} at $L=1$.
It is particularly useful for applications which do not admit the exact form required by Theorem~\ref{thm:pL_rankM_integral}, as stated by Lemma~\ref{lem:symmetric_factor_extension_to_Theorem_2}.
\begin{lemma}[Symmetric factor extension to Theorem~\ref{thm:pL_rankM_integral}]
    \label{lem:symmetric_factor_extension_to_Theorem_2} Let $\theta_{lk,i}\in \mathbb R^+$ and $a_{lk,i}\in\{\pm 1\}$ for all $(l,k,i)$, $M(N)=O\left(N^{1/2-\epsilon}\right)$,
    $\Theta_{lk} = \mathrm{diag}\big\{\theta_{lk,1},\ldots, \theta_{lk,M(N)},0,\ldots,0\big\}\in\mathbb R^{N\times N}$, and $A_{lk}=\mathrm{diag}\big\{a_{lk,1},\ldots, a_{lk,M(N)},0,\ldots,0\big\}\in \mathbb R^{N\times N}$ with $A_{lk}=A_{kl}$ and $\Theta_{lk}=\Theta_{kl}$.
    Define an $L$-fold, coupled spherical integral of the form
    \begin{align*}
        \hat I_N^{(\beta)}
        \left(
            \{\Theta_{lk}\},\{A_{lk}\},\{E_{lk}\}
        \right)
        =
        \int
            \exp\left\{
                \sum_{l,k=1}^L
                    N \mathrm{tr}\left(
                        A_{lk}\Theta_{lk}  U_l^\dagger E_{lk} U_k
                    \right)
            \right\}
            \prod_{l=1}^L 
                \mathrm d m_N^{(\beta)}(U_l),
    \end{align*}
    and let the following block matrix $C_{N,i}$ converge to some $C_{i}$ for all $i\in[M(N)]$ in the sense of Definition~\ref{def:operator_valued_convergence}
    \begin{align*}
        C_{N,i}
        =
        \begin{bmatrix}
            a_{11,i} E_{11} 
            & a_{12,i} \frac{\theta_{12,i}}{\sqrt{\theta_{11,i}\theta_{22,i}}} E_{12}
            & \ldots
            \\
            a_{21,i} \frac{\theta_{21,i}}{\sqrt{\theta_{11,i}\theta_{22,i}}}E_{21}
            & a_{22,i} E_{22}  
            \\
            \vdots && \ddots
        \end{bmatrix}
        \in \mathbb A^{LN\times LN}.
    \end{align*}
    Furthermore, all moments of $C_{i}$ are finite, $C_{i}$ is R-cyclic and self-adjoint, and $||E_{lk}||_\infty$ is finitely bounded for all $(l,k)$ as $N\uparrow\infty$.
    If $M(N)=O\left(N^{1/2-\epsilon}\right)$ for any $\epsilon>0$, $L=O(1)$ and all $\theta_{lk,i}$ are small enough such that
    \begin{align*}
        -P_i = -\mathrm{diag}\left\{
            \frac{2\theta_{11,i}}{\beta},\frac{2\theta_{22,i}}{\beta}, \ldots, \frac{2\theta_{LL,i}}{\beta}
        \right\}
        \in \left\{
            \left.
            G_{C_{i}}^{\mathcal D_L}(S)
            \right|
            S\in\mathcal S_{\{-\}^L}^{C_i}
        \right\},
    \end{align*}
    then
    \begin{align*}
        \lim_{N\uparrow\infty}
        \frac{1}{NM(N)} \log \hat I_N^{(\beta)}\left(
            \{\Theta_{lk}\},\{A_{lk}\},\{E_{lk}\}
        \right)
        =
        \frac{1}{M(N)}
        \sum_{i=1}^{M(N)}
        \frac{\beta}{2}
        \int_0^1
            \mathrm{tr}\left(
                R^{\mathcal D_L}_{C_{i}}(P_i w) P_i
            \right)
            \mathrm d w.
    \end{align*}
\end{lemma}
The proof follows from Theorem~\ref{thm:pL_rankM_integral_modification}:
\begin{proof}[Proof of Lemma~\upshape{\ref{lem:symmetric_factor_extension_to_Theorem_2}}]
    Let $\Theta_{lk,i}\in\mathbb R^{N\times N}$ have a single non-zero entry at position $(i,i)$ denoted by $\theta_{lk,i}$.
    Note that $\Theta_{lk} = \sum_{i}\Theta_{lk,i}$, such that
    \begin{align}
        \hat I_N^{(\beta)}
        \left(
            \{\Theta_{lk}\}, \{A_{lk}\}, \{E_{lk}\}
        \right)
        &=
        \int
            \exp\left\{
                \sum_{\substack{i=1\\\phantom{k\neq l}}}^{M(N)}
                \sum_{l=1}^L
                    N \mathrm{tr}\left(
                        \Theta_{ll,i}  U_l^\dagger a_{ll,i}E_{ll} U_l
                    \right)
            \right.
            \\
            &~~~~~~\left.
                +
                \sum_{\substack{k=1\\k\neq l}}^{L}
                    N \mathrm{tr}\left(
                        \Theta_{ll,i}^{\frac{1}{2}}\Theta_{kk,i}^{\frac{1}{2}}  U_l^\dagger \frac{a_{lk,i}\theta_{lk,i}}{\sqrt{\theta_{ll,i}\theta_{kk,i}}}E_{lk} U_k
                    \right)
            \right\}
            \prod_{l=1}^L 
                \mathrm d m_N^{(\beta)}(U_l)
            \nonumber
            \\
        &= J_N^{(\beta)}\left(\left\{\Theta_{ll,i}\right\}, \left\{\hat E_{lk,i}\right\}\right),
    \end{align}
    where
    \begin{align}
        \hat E_{lk,i} = \frac{a_{lk,i} \theta_{lk,i}}{\sqrt{\theta_{ll,i}\theta_{kk,i}}} E_{lk}.
    \end{align}
\end{proof}
In this paper we only consider the high-temperature regime in which all $\theta_{l,i}$ are small enough such that the corresponding R-transform exists.
For rank-one HCIZ integrals, this constraint was lifted in \cite[Theorem 6]{guionnetFourierViewTransform2005}.
Recent results in \cite{guionnetAsymptoticsDimensionalSpherical2021} and \cite{hussonSphericalIntegralsSublinear2025} have extended these results to a finite number of non-zero eigenvalues ($M(N)=k=O(1)$) and extensive rank ($M(N)=o(N)$), respectively, assuming that the extremal eigenvalues of $E_N$ lie outside the support of $\mu_E$.
The extensive rank case in which the extremal eigenvalues do not violate the support of $\mu_E$ was studied in \cite{collinsNewScalingItzykson2007}.
Corresponding analysis of the $L$-fold, coupled spherical integral are subject to future work.
\section{Preliminaries}\label{3_preliminaries}
\subsection{Scalar Stieltjes transform}
The Stieltjes transform  $G_Z:\mathbb R\to \mathbb R$ of a probability measure $\mu_Z$ with support $\mathrm{supp}(\mu_Z) = [\lambda_\mathrm{min}(Z), \lambda_{\mathrm{max}}(Z)]$ on the real line is defined by
\begin{align}
    G_Z(s) 
    &= \int \frac{1}{z-s} d \mu_Z(z).
\end{align}
Let $Z=\lim_{N\uparrow\infty} Z_N$ be a unitarily invariant $N\times N$ random matrix with spectral measure $\mu_Z$, then $G_Z(s)$ is referred to as the Stieltjes transform of $Z$.
The Stieltjes transform admits a number of properties. We only restate the domain and monotonically increasing properties as they are relevant for the main results of this work.
See \cite{guionnetFourierViewTransform2005} for an exhaustive discussion of the Hilbert transform on the real line which is closely related to $G_Z(s)$.
\begin{prp}[Existence on the real line]
    \label{prop:scalar_stieltjes_transform_existence}
    $G_Z(s)$ exists on the real line for $s<\lambda_\mathrm{min}(Z)$ or $s>\lambda_\mathrm{max}(Z)$, where $\lambda_\mathrm{min}(Z)$ and $\lambda_\mathrm{max}(Z)$ are the minimum and maximum eigenvalues of $Z$, respectively.
\end{prp}
\begin{prp}[Monotonically increasing on the real line]
    \label{prop:scalar_stieltjes_transform_monotonic_increasing}
    The Stieltjes transform of $Z$ is positive and monotonically increasing for all $s<\lambda_\mathrm{min}(Z)$ and negative and monotonically increasing for all $s>\lambda_\mathrm{max}(Z)$.
    As $s\downarrow-\infty$, $G_Z(s)$ tends to zero from above, and as $s\uparrow\infty$, $G_Z(s)$ tends to zero from below.
    $G_Z(s)$ is hence bijective from $\mathbb R \setminus [\lambda_\mathrm{min}(Z),\lambda_\mathrm{max}(Z)]$ to $(G_Z^\mathrm{min},G_Z^\mathrm{max})\setminus\{0\}$ with $G_Z^\mathrm{max}=\lim_{s\uparrow\lambda_\mathrm{min}(Z)} G_Z(s)$, and $G_Z^\mathrm{min}=\lim_{s\downarrow\lambda_\mathrm{max}(Z)} G_Z(s)$.
\end{prp}
From Property \ref{prop:scalar_stieltjes_transform_monotonic_increasing} it is clear that the inverse of the Stieltjes transform with respect to composition $K_Z=G_Z^{-1}$ exists, is unique on the domain $(G_Z^\mathrm{min},G_Z^\mathrm{max})\setminus\{0\}$, and is monotonically increasing on this domain with a pole around $0$.
Denote the empirical Stieltjes transform of $Z_N$ by $\mathfrak{G}_{Z_N}(s)$.
We say that $Z_N$ converges to $Z$ if
\begin{align}
    G_Z(s) = \lim_{N\uparrow\infty} \mathfrak{G}_{Z_N}(s) = \lim_{N\uparrow\infty}\mathrm{Tr}\left[(Z_N-s \mathbb I_N)^{-1}\right].
\end{align}
If the spectrum of $Z_N$ is lower and upper bounded by $\lambda_\mathrm{min}(Z)$ and $\lambda_\mathrm{max}(Z)$, respectively, for all $N>N_0$ and some $N_0>0$, Properties~\ref{prop:scalar_stieltjes_transform_existence}~and~\ref{prop:scalar_stieltjes_transform_monotonic_increasing} extend to $\mathfrak{G}_{Z_N}(s)$ for all $N>N_0$.

\subsection{Operator-Valued Stieltjes transform}
Let $S = \{s_{lk}\}_{l,k=1}^L$ with $ s_{lk}\in \mathbb A$, and $Z_N \in \mathbb A^{LN\times LN}$ is an $L\times L$ self-adjoint block matrix with blocks $E_{lk}\in \mathbb A^{N\times N}$.
Assume that $Z_N$ converges to $Z$ in the sense of Definition~\ref{def:operator_valued_convergence}.
The operator-valued Stieltjes transform $G_Z^{\mathcal B_L}: \mathcal B_L \to \mathcal B_L$ on the algebra of $L\times L$ matrices $\mathcal B_L$ is then defined by the asymptote of the corresponding empirical transform $\mathfrak{G}_{Z_N}^{\mathcal B_L}$
\begin{align}
    G_Z^{\mathcal B_L}(S) 
    =
    \lim_{N\uparrow\infty}
    \mathfrak{G}_{Z_N}^{\mathcal B_L}(S),
    \label{eq:definition_operator_valued_stieltjes_transform}
\end{align}
with
\begin{align}
    \label{eq:definition_empirical_operator_valued_stieltjes_transform}
    \mathfrak{G}_{Z_N}^{\mathcal B_L}(S)
    = 
    \mathrm{Tr}_L\left[\left(Z_N-S\otimes \mathbb I_N\right)^{-1}\right]
    =
    -\sum_{n=0}^\infty \mathrm{Tr}_L\left[S^{-1}\left(Z_NS^{-1}\right)^n \right].
\end{align}
The power series does exist if $||S||_\infty>||Z_N||_\infty$, i.e. $(Z_N-S\otimes \mathbb I_N)^{-1}$ exists \cite{mingoFreeProbabilityRandom2017}.
Consider the operator-valued probability space $(\mathcal A, \mathcal E, \mathcal B_L)$, then the operator-valued Stieltjes transform on $\mathcal B_L$ is, furthermore, defined by a corresponding power series of the moments
\begin{align}
    G_Z^{\mathcal{B}_L}(S) = -\sum_{n=0}^\infty
    \mathcal E\left[S^{-1}\left(ZS^{-1}\right)^{n}\right],
\end{align}
where $\mathcal E:\mathcal A\to\mathcal B_L$ is the corresponding expectation operation.
Furthermore, denote the functional inverse of the operator-valued Stieltjes transform $G_Z^{\mathcal B_L}(S)$ as $K_Z^{\mathcal B_L}(W)$ and its empirical counterpart $\mathfrak K_{Z_N}^{\mathcal B_L}(W)$.
Based on the power series representation, it is straightforward to derive the following left and right scaling properties:
\begin{prp}[Left Scaling of Z]
    \label{prop:operator_valued_stieltjes_transform_left_scaling}
    Let $A\in \mathbb A^{L\times L}$ and its inverse exists, consider $G_Z^{\mathcal B_L}$ and its functional inverse $K_Z^{\mathcal B_L}$, then
    \begin{align}
        G^{\mathcal B_L}_{ZA}(S) &= A^{-1} G_Z^{\mathcal B_L}(SA^{-1}),
        \\
        K^{\mathcal B_L}_{ZA}(W) &= K_Z^{\mathcal B_L}(AW)A.
    \end{align}
\end{prp}
\begin{prp}[Right Scaling of Z]
    \label{prop:operator_valued_stieltjes_transform_right_scaling}
    Let $A\in \mathbb A^{L\times L}$ and its inverse exists, consider $G_Z^{\mathcal B_L}$ and its functional inverse $K_Z^{\mathcal B_L}$ as in Property~\ref{prop:operator_valued_stieltjes_transform_left_scaling}, then
    \begin{align}
        G^{\mathcal B_L}_{AZ}(S) &=  G^{\mathcal B_L}_Z(A^{-1}S)A^{-1},
        \\
        K^{\mathcal B_L}_{AZ}(W) &= AK^{\mathcal B_L}_Z(WA).
    \end{align}
\end{prp}
Consider the operator-valued probability space $(\mathcal A, \mathcal E_{\mathcal D}, \mathcal D_L)$.
The corresponding Stieltjes transform is given as
\begin{align}
    G_Z^{\mathcal{D}_L}(S) = -\sum_{n=0}^\infty
    \mathcal E_{\mathcal D}\left[S^{-1}\left(ZS^{-1}\right)^{n}\right],
\end{align}
with $S=\mathrm{diag}\left\{s_1,\ldots,s_L\right\}$.
If $Z$ is R-cyclic, then $G_Z^{\mathcal B_L}(S)=G_Z^{\mathcal D_L}(S)$ for all $S\in\mathcal{D}_L$ \cite{mingoFreeProbabilityRandom2017}, and \eqref{eq:definition_empirical_operator_valued_stieltjes_transform} converges to $G_Z^{\mathcal D_L}(S)$ with all off-diagonal traces tending to zero as $N\uparrow\infty$ if $S\in\mathcal D_L$.
Throughout this work we assume $Z$ to be R-cyclic and $S\in \mathbb R^{L\times L}$ to be a real, diagonal matrix. 
We refer to $G^{\mathcal D_L}_Z(S)$ as the $L$-conditional Stieltjes transform.
The existence of $G_Z^{\mathcal D_L}(S)$ is guaranteed for any $S$ in the complex-valued upper half-plane \cite{mingoFreeProbabilityRandom2017}.
The main results of this paper, however, consider the real plane $s_l\in\mathbb R ~\forall~l$.
Properties~\ref{prop:2conditional_stieltjes_transform_on_real_plane} and \ref{prop:Lconditional_stieltjes_transform_on_real_plane} characterize the existence of the operator-valued Stieltjes transform on $\mathcal D_2$ and $\mathcal D_L$ for any $L>2$, respectively.
For $L>2$ we define the domain by hierarchically splitting down to $\mathcal D_2$, and only consider the negative definite case, hence the two cases are separated into two properties.
\begin{prp}[Stieltjes transform of $\mathcal D_2$ on the Real Plane]
\label{prop:2conditional_stieltjes_transform_on_real_plane}
Consider the Stieltjes transform $G_Z^{\mathcal D_2}(S)$ on the algebra of diagonal $2\times 2$ matrices $\mathcal D_2$.
Let $S=\mathrm{diag}\{s_1,s_2\}\in\mathbb R^{2\times 2}$ and 
\begin{align*}
    Z_N = \begin{bmatrix}
        E_{11} & E_{12} 
        \\
        E_{12}^\dagger & E_{22}
    \end{bmatrix}
    = Z_N^\dagger,
\end{align*}
where the spectrum of $E_{ll}$ is bounded by $\lambda_\mathrm{min}(E_{ll})$ and $\lambda_\mathrm{max}(E_{ll})$ and the maximum singular value of $E_{12}$ is bounded by $||E_{12}||_\infty$ as $N\uparrow\infty$.
Furthermore, $Z_N$ converges to $Z$ in the sense of Definition~\ref{def:operator_valued_convergence}.
Then, the following statements hold:
\begin{enumerate}
    \item[(1)] The operator-valued Stieltjes transform exists on the domain $\mathcal S_{\mathcal{D}_2}^{Z}  = \mathcal S_{--}^Z \cup \mathcal S_{-+}^Z \cup \mathcal S_{+-}^Z\cup\mathcal S_{++}^Z$, with
    \begin{align*}
        \mathcal S_{--}^Z
        &= \Bigg\{ 
            \mathrm{diag}\{s_1,s_2\} \left| s_1,s_2\in\mathbb R \land
            s_2>\lambda_\mathrm{max} (E_{22}) +\frac{||E_{12}||_\infty^2}{s_1-\lambda_\mathrm{max}(E_{11})}
            \right. \land s_l>\lambda_\mathrm{max}(E_{ll})~\forall ~l\in\{1,2\}
        \Bigg\},
        \\
        \mathcal S_{-+}^Z
        &=\left\{ 
            \mathrm{diag}\{s_1,s_2\} \left| s_1,s_2\in\mathbb R \land
            s_1>\lambda_\mathrm{max}(E_{11}) \land s_2<\lambda_\mathrm{min}(E_{22})
            \right.
        \right\},
        \\
        \mathcal S_{+-}^Z
        &=\left\{ 
            \mathrm{diag}\{s_1,s_2\} \left| s_1,s_2\in\mathbb R \land
            s_1<\lambda_\mathrm{min}(E_{11}) \land s_2>\lambda_\mathrm{max}(E_{22})
            \right.
        \right\},
        \\
        \mathcal S_{++}^Z
        &=\Bigg\{ 
            \mathrm{diag}\{s_1,s_2\} \left| s_1,s_2\in\mathbb R \land
            s_2<\lambda_\mathrm{min} (E_{22}) +\frac{||E_{12}||_\infty^2}{s_1-\lambda_\mathrm{min}(E_{11})}
            \right.
            \land s_l<\lambda_\mathrm{min}(E_{ll})~\forall ~l\in\{1,2\}
        \Bigg\},
    \end{align*}
    as illustrated in Fig. \ref{fig:2ConditionalStieltjesOnRealPlane}.
    \begin{figure}
        \centering
        \begingroup%
        \makeatletter%
        \providecommand\color[2][]{%
            \errmessage{(Inkscape) Color is used for the text in Inkscape, but the package 'color.sty' is not loaded}%
            \renewcommand\color[2][]{}%
        }%
        \providecommand\transparent[1]{%
            \errmessage{(Inkscape) Transparency is used (non-zero) for the text in Inkscape, but the package 'transparent.sty' is not loaded}%
            \renewcommand\transparent[1]{}%
        }%
        \providecommand\rotatebox[2]{#2}%
        \newcommand*\fsize{\dimexpr\f@size pt\relax}%
        \newcommand*\lineheight[1]{\fontsize{\fsize}{#1\fsize}\selectfont}%
        \ifx\svgwidth\undefined%
            \setlength{\unitlength}{198.42521848bp}%
            \ifx\svgscale\undefined%
            \relax%
            \else%
            \setlength{\unitlength}{\unitlength * \real{\svgscale}}%
            \fi%
        \else%
            \setlength{\unitlength}{\svgwidth}%
        \fi%
        \global\let\svgwidth\undefined%
        \global\let\svgscale\undefined%
        \makeatother%
        \begin{picture}(1,1.04564307)%
            \lineheight{1}%
            \setlength\tabcolsep{0pt}%
            \put(0,0){\includegraphics[width=\unitlength,page=1]{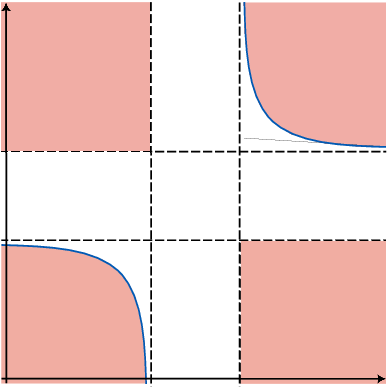}}%
            \put(0.75088997,0.24232938){\color[rgb]{0.28235294,0.38823529,1}\makebox(0,0)[lt]{\lineheight{1.25}\smash{\begin{tabular}[t]{l}$$\end{tabular}}}}%
            \put(0.21515514,0.19753667){\color[rgb]{0,0,0}\makebox(0,0)[t]{\lineheight{1.25}\smash{\begin{tabular}[t]{c}$\mathcal S_{++}^Z$\end{tabular}}}}%
            \put(0.7907178,0.77218914){\color[rgb]{0,0,0}\makebox(0,0)[t]{\lineheight{1.25}\smash{\begin{tabular}[t]{c}$\mathcal S_{--}^Z$\end{tabular}}}}%
            \put(0.21519318,0.76866425){\color[rgb]{0,0,0}\makebox(0,0)[t]{\lineheight{1.25}\smash{\begin{tabular}[t]{c}$\mathcal S_{+-}^Z$\end{tabular}}}}%
            \put(0.78910146,0.19917546){\color[rgb]{0,0,0}\makebox(0,0)[t]{\lineheight{1.25}\smash{\begin{tabular}[t]{c}$\mathcal S_{-+}^Z$\end{tabular}}}}%
            \put(0.05051479,1.00266508){\color[rgb]{0,0,0}\makebox(0,0)[t]{\lineheight{1.25}\smash{\begin{tabular}[t]{c}$s_2$\end{tabular}}}}%
            \put(0.39086483,1.01253195){\color[rgb]{0,0,0}\makebox(0,0)[t]{\lineheight{1.25}\smash{\begin{tabular}[t]{c}$\lambda_\mathrm{min}(E_{11})$\end{tabular}}}}%
            \put(0.62512948,1.01253195){\color[rgb]{0,0,0}\makebox(0,0)[t]{\lineheight{1.25}\smash{\begin{tabular}[t]{c}$\lambda_\mathrm{max}(E_{11})$\end{tabular}}}}%
            \put(-0.001,0.37316073){\color[rgb]{0,0,0}\makebox(0,0)[rt]{\lineheight{1.25}\smash{\begin{tabular}[t]{r}$\lambda_\mathrm{min}(E_{22})$\end{tabular}}}}%
            \put(-0.001,0.58742544){\color[rgb]{0,0,0}\makebox(0,0)[rt]{\lineheight{1.25}\smash{\begin{tabular}[t]{r}$\lambda_\mathrm{max}(E_{22})$\end{tabular}}}}%
            \put(1.0067653,0.05135519){\color[rgb]{0,0,0}\makebox(0,0)[lt]{\lineheight{1.25}\smash{\begin{tabular}[t]{l}$s_1$\end{tabular}}}}%
            \put(0.99803178,0.60996245){\color[rgb]{0,0.35294,0.7098}\makebox(0,0)[lt]{\lineheight{1.25}\smash{\begin{tabular}[t]{l}$s_2 = \lambda_\mathrm{max}(E_{22})$\\~~~~~~$\displaystyle+\frac{||E_{12}||_\infty^2}{s_1-\lambda_\mathrm{max}(E_{11})}$\end{tabular}}}}%
            \put(0.35522621,-0.07088314){\color[rgb]{0,0.35294,0.7098}\makebox(0,0)[lt]{\lineheight{1.25}\smash{\begin{tabular}[t]{l}$\displaystyle s_2 = \lambda_\mathrm{min}(E_{22})+\frac{||E_{12}||_\infty^2}{s_1-\lambda_\mathrm{min}(E_{11})}$\end{tabular}}}}%
        \end{picture}%
        \endgroup%
        \vspace*{10mm}
        \caption{Domain of the operator-valued Stieltjes transform of $Z$ on the algebra of real, diagonal $2\times 2$ matrices $\mathcal D_2$}
        \label{fig:2ConditionalStieltjesOnRealPlane}
    \end{figure}
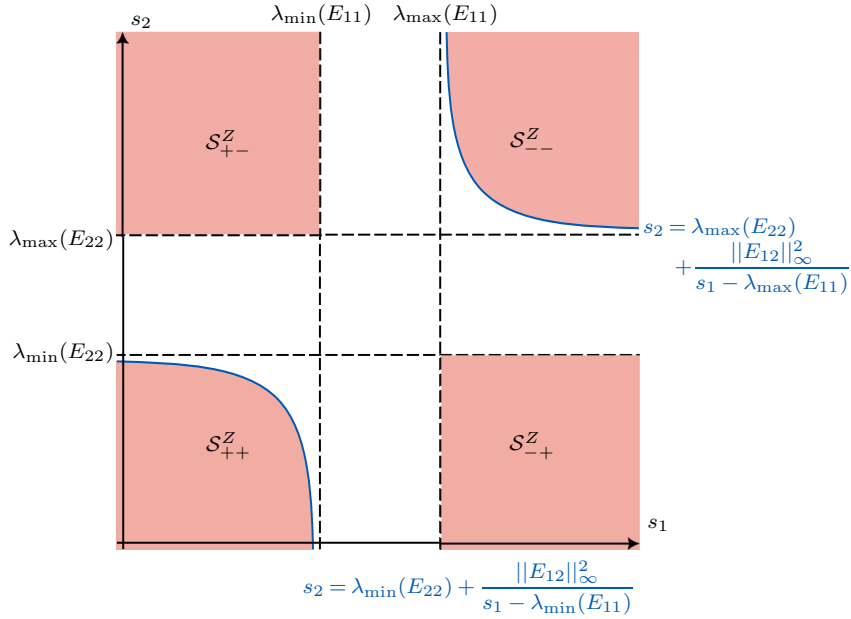
    \item[(2)] $Z-S\otimes \mathbb I\prec 0$ for all $S\in\mathcal S_{--}^Z$ and $Z-S\otimes \mathbb I \succ 0$ for all $S\in\mathcal S_{++}^Z$.
    \item[(3)] The subdomains $\mathcal S_{pm}^Z$ with $p,m\in\{+,-\}$ do not overlap.
    \item[(4)] For $S\in \mathcal S_{pm}^Z$ the sign of $(G_{Z}(S))_{11}$ and $(G_Z(S))_{22}$ is $p$ and $m$, respectively.
    \item[(5)] For $S \in \mathcal S_{\mathcal D_2}^Z$, $(G_Z(S))_{ll}$ is increasing for $s_1$ and $s_2$ increasing for all $l\in\{1,2\}$.
    \item[(6)] $G_Z^{\mathcal D_2}(S)$ is bijective on $\mathcal S_{--}^Z \cup \mathcal S_{++}^Z$, hence the inverse with respect to composition $K_Z^{\mathcal D_2}(W)$ exists for all $W\in\mathcal W_{--}^Z\cup \mathcal W_{++}^Z$, with
    \begin{align*}
        \mathcal W_{pp}^Z = \left\{\left.
            G_Z^{\mathcal D_2}(S)\right|S\in \mathcal S_{pp}^Z
        \right\}, \text{ for }p\in\{\pm\},
    \end{align*}
\end{enumerate}
where $(X)_{lk}$ identifies element $(l,k)$ of matrix $X$.
All properties hold by extension for the empirical Stieltjes transform $\mathfrak{G}_{Z_N}^{\mathcal D_2}$ and its functional inverse $\mathfrak{K}_{Z_N}^{\mathcal D_2}$ for $N>N_0$, assuming that all required operator norms and minimum and maximum eigenvalues are finitely bounded for all $N>N_0$ and some $N_0>0$.
\end{prp}
Note that the existence of a functional inverse on $\mathcal{S}_{pm}^{Z}$ for $p\neq m$ is not guaranteed. 
In fact, a simple counterexample can be constructed by $E_{11}=E_{22}=0$ and $E_{12}=E_{21}=\mathbb I_N$, such that $G_Z^{\mathcal D_2}(S)=\mathrm{diag}\left\{-2/5, 2/5\right\}$ for both $S=\mathrm{diag}\{1/2, -1/2\}$ and $S=\mathrm{diag}\{2, -2\}$.
This effect occurs because $Z-S\otimes \mathbb I_N$ is non-definite on $\mathcal{S}_{pm}^{Z}$ for $p\neq m$, i.e. the Jacobian of $G_Z^{\mathcal D_2}(S)$ is non-definite, as discussed in the proof of Property~\ref{prop:2conditional_stieltjes_transform_on_real_plane}.
Theorems~\ref{thm:pL_rank1_integral}-\ref{thm:pL_rankM_integral_modification} only consider $p=m$, such that this is not an issue for our work.
A more precise characterization of the operator-valued Stieltjes transform on the complete domain is subject to future research.
For $L>2$ the definition of the domain $\mathcal S_{\mathcal D_L}^Z$ is more challenging. 
Hence, only the negative subdomain $\mathcal S_{\{-\}^L}^Z$, with $\mathcal S_{--}^Z=\mathcal S_{\{-\}^2}^Z$, is characterized for any $L>2$.
\begin{prp}[Stieltejes Transform of $\mathcal D_L$ on the Real Plane for $L>2$]
\label{prop:Lconditional_stieltjes_transform_on_real_plane}
    Let $Z_N$ be defined as in Definition~\ref{def:operator_valued_convergence}, and $Z_N$ converges to $Z$. 
    The spectra of $E_{ll}$ and $Z$ are bound by their respective minimum and maximum eigenvalues $\lambda_\mathrm{min}(\cdot)$ and $\lambda_\mathrm{max}(\cdot)$, the required operator norms $||\cdot||_\infty$ are finite as $N\uparrow\infty$.
    Let $Z_{\setminus l}\in \mathbb R^{(L-1)N\times(L-1)N}$ be the matrix which is constructed by omitting the $l^\mathrm{th}$ column and row of $Z_N$, $Q_l\in\mathbb R^{(L-1)N\times N}$ the $l^\mathrm{th}$ block-column of $Z_N$ in which the $l^\mathrm{th}$ block is omitted, and $S_{\setminus l}\in \mathbb R^{(L-1)\times(L-1)}$ the sub-matrix of $S$ where the $l^\mathrm{th}$ column and row are omitted.
    Furthermore, define $\hat Z^{(l)}$ as
    \begin{align}
        \hat Z^{(l)} = Z_{\setminus l} + Q_l (s_l \mathbb I_N -E_{ll})^{-1} Q_l^T.
    \end{align}
    Note that we dropped the subscript $N$ for brevity and use the notation for both the finite dimensional matrices and the ones as $N\uparrow\infty$. 
    Then the following statements hold:
    \begin{enumerate}
        \item[(1)] The \(L\)-conditional Stieltjes transform exists on the domain
        \begin{align*}
            \mathcal S_{\{-\}^L}^Z 
            =
                \Bigg\{&
                    \mathrm{diag}\{s_1,\ldots,s_L\}
                \Bigg |
                    s_l\in\mathbb R \land
                    s_l 
                        > 
                            \lambda_\mathrm{max}(E_{ll}) 
                            + 
                            ||Q_l||^2_\infty 
                            \left|\left|
                                Z_{\setminus l}
                                - 
                                S_{\setminus l}
                                \otimes 
                                \mathbb I
                            \right|\right|^{-1}_\infty
                    \land ~S_{\setminus l} \in  \mathcal S_{\{-\}^{L-1}}^{\hat Z^{(l)}}
                \Bigg\}.
        \end{align*}
        \item[(2)] For $S\in\mathcal S_{\{-\}^L}^Z$, $Z-S\otimes \mathbb I\prec 0$ is negative definite and, hence, the $L$-conditional Stieltjes transform is negative definite $G_Z^{\mathcal D_L}(S)\prec 0$. 
        \item[(3)] $G_Z^{\mathcal D_L}(S)$ is increasing with $s_l$ increasing for any $l$.
        \item[(4)] $G_Z^{\mathcal D_L}(S)$ is bijective on $S_{\{-\}^L}^Z$, hence the inverse with respect to composition $K_Z^{\mathcal D_L}(W)$ exists for all $W \in\mathcal W_{\{-\}^L}^Z$, with
        \begin{align*}
            \mathcal W_{\{-\}^L}^Z = \left\{
                \left.
                G_Z^{\mathcal D_L}(S) \right| S\in\mathcal S_{\{-\}^L}^Z
            \right\}.
        \end{align*}
    \end{enumerate}
    Property~\ref{prop:Lconditional_stieltjes_transform_on_real_plane} holds by extension for the empirical transform $\mathfrak{G}_{Z_N}^{\mathcal D_L}$ and its functional inverse $\mathfrak{K}_{Z_N}^{\mathcal D_L}$ for $N>N_0$, assuming that the required operator norms and minimum and maximum eigenvalues are finitely bounded for all $N>N_0$ and some $N_0>0$.
\end{prp}
The proofs of Properties~\ref{prop:2conditional_stieltjes_transform_on_real_plane} and \ref{prop:Lconditional_stieltjes_transform_on_real_plane} are provided in Section~\ref{4_proofs}.

\subsection{Operator-Valued R-transform}
The operator-valued R-transform on the algebra of $L\times L$ diagonal, real matrices is defined by
\begin{align}
    R_Z^{\mathcal D_L}(W) = K_Z^{\mathcal{D}_L}(-W)-W^{-1}~\forall~-W\in \mathcal W_{\{-\}^L}^{Z}.
\end{align}
Alternatively, the R-transform is expressed by a power series of operator-valued cumulants \cite{mingoFreeProbabilityRandom2017}, as
\begin{align}
    R_Z^{\mathcal D_L}(W) = \sum_{n=1}^\infty
        \mathcal \kappa^{\mathcal D_L}_n(ZW, \ldots, ZW, Z).
\end{align}
From this we can directly see that if $R_Z^{\mathcal D_L}(W)$ exists, $R_Z^{\mathcal D_L}(wW)~\forall w\in[0,1)$ exists, since the power series still converges.
The empirical operator-valued R-transform is defined accordingly through the functional inverse of ${\mathfrak{G}_{Z_N}^{\mathcal D_L}}$, which we assume to exist for large enough $N$, as
\begin{align}
    \mathfrak{R}_{Z_N}^{\mathcal D_L}(W) = \mathfrak{K}_{Z_N}^{\mathcal D_L}(-W)-W^{-1}~\forall~-W\in\mathcal W_{\{-\}^L}^{Z}.
\end{align}

\section{Proofs}\label{4_proofs}
The proofs of Properties \ref{prop:operator_valued_stieltjes_transform_left_scaling} and \ref{prop:operator_valued_stieltjes_transform_right_scaling} are straightforward based on the power series definition of the operator-valued Stieltjes transform, and hence are omitted for brevity.

\begin{proof}[Proof of Property~\upshape{\ref{prop:2conditional_stieltjes_transform_on_real_plane}}]
    Consider the block matrix $Z_N-S\otimes \mathbb I_N$ with $S=\mathrm{diag}\{s_1,s_2\}$ and let $N>N_0$ with $N_0>0$ such that the minimum (maximum) eigenvalue of $E_{ll}$ is finitely bounded by $\lambda_\mathrm{min}(E_{ll})$ ($\lambda_\mathrm{max}(E_{ll})$), and the maximum singular value of $E_{lk}$ is finitely bounded by $||E_{lk}||_\infty$. The block matrix inverse is given by
    \begin{align}
        (Z_N-S\otimes \mathbb I_N)^{-1}
        =
        \begin{bmatrix}
            E_{11}-s_1\mathbb I_N & E_{12} \\
            E_{12}^\dagger & E_{22}-s_2\mathbb I_N
        \end{bmatrix}^{-1}
        =
        \begin{bmatrix}
            \left(\tilde E_{11} - E_{12}{\tilde E_{22}}^{-1}E_{21}\right)^{-1} & \square_{N\times N} \\
            \square_{N\times N} & \left(\tilde E_{22} - E_{21}{\tilde E_{11}}^{-1}E_{12}\right)^{-1}
        \end{bmatrix},
        \label{eq:proof_prop_5_block_mat_inverse}
    \end{align}
    where $\square_{n\times m}$ denotes any matrix of size $n\times m$, $\tilde{E}_{ll} = E_{ll}-s_l\mathbb I_N$, and assuming that $\tilde E_{ll}$ is invertible, i.e. the empirical scalar Stieltjes transform $\mathfrak G_{E_{ll}}(s_l)$ exists for all $l\in\{1,2\}$.
    We consider the 2-conditional Stieltjes transform on $\mathcal D_2$, and hence ignore the off-diagonal entries of the inverse matrix.
    The transform exists if the inverse on the diagonal entries exists.
    We, therefore, investigate the following four cases:
    \begin{enumerate}
        \item[(c 1)] $s_1>\lambda_\mathrm{max}(E_{11})$ and $s_2>\lambda_\mathrm{max}(E_{22})$,
        \item[(c 2)] $s_1>\lambda_\mathrm{max}(E_{11})$ and $s_2<\lambda_\mathrm{min}(E_{22})$,
        \item[(c 3)] $s_1<\lambda_\mathrm{min}(E_{11})$ and $s_2>\lambda_\mathrm{max}(E_{22})$,
        \item[(c 4)] $s_1<\lambda_\mathrm{min}(E_{11})$ and $s_2<\lambda_\mathrm{min}(E_{22})$.
    \end{enumerate}
    Consider case (c 1), $\tilde{E}_{22}^{-1}$ is negative definite, such that a decomposition $\tilde{E}_{22}^{-1}=-B^\dagger B$ exists which shows that $-E_{12}\tilde E_{22}^{-1} E_{21}$ is positive semi-definite.
    The maximum eigenvalue of this matrix product is given by $\lambda_\mathrm{max}\left(-E_{12}\tilde E_{22}^{-1} E_{21}\right)\leq\frac{||E_{12}||_\infty^2}{s_2-\lambda_\mathrm{max}(E_{22})}$.
    We conclude that the inverse of the first diagonal entry exists if $s_1>\lambda_\mathrm{max}(E_{11}) +\frac{||E_{12}||_\infty^2}{s_2-\lambda_\mathrm{max}(E_{22})}$.
    The inverse of the second diagonal entry exists if $s_2>\lambda_\mathrm{max}(E_{22}) +\frac{||E_{12}||_\infty^2}{s_1-\lambda_\mathrm{max}(E_{11})}$, by similar argument.
    It is straightforward to see that the two conditions are identical, and the Stieltjes transform exists for $S\in \mathcal S_{--}^Z$.
    In case (c 2), $\tilde{E}_{22}^{-1}$ is positive definite, and $-E_{12}\tilde E_{22}^{-1} E_{21}$ negative semi-definite.
    The maximum eigenvalue of $E_{11}-E_{12}\tilde E_{22}^{-1}E_{21}$ is, therefore, upper bounded by $\lambda_\mathrm{max}(E_{11})$.
    Similarly, the minimum eigenvalue of $E_{22}-E_{21}\tilde E_{11}^{-1}E_{12}$ is lower bounded by $\lambda_\mathrm{min}(E_{22})$.
    The Stieltjes transform exists for $S\in \mathcal S_{-+}^Z$, and similarly also for $S\in \mathcal S_{+-}^Z$.
    Finally, in case (c 4), $\tilde{E}_{22}^{-1}$ is positive definite, and $-E_{12}\tilde E_{22}^{-1} E_{21}$ negative semi-definite.
    We obtain the bound
    \begin{align}
        \lambda_\mathrm{min}\left(E_{11}-E_{12}\tilde E_{22}^{-1}E_{21}\right)
        \geq 
        \lambda_\mathrm{min}(E_{11})
        +
        \frac{||E_{12}||_\infty^2}{
            s_2-\lambda_\mathrm{min}(E_{22})
        },
    \end{align}
    which shows that the Stieltjes transform exists for $S\in \mathcal S_{++}^Z$.
    By construction, the subdomains $\mathcal S_{pm}^Z$ with $p,m\in\{+,-\}$ do not overlap.
    For $S\in \mathcal S_{pm}^Z$ the sign of $[G_Z(S)]_{11}$ and $[G_Z(S)]_{22}$ is $p$ and $m$, respectively, since the partial trace of order 2 of $(Z_N-S\otimes \mathbb I_N)^{-1}$ produces scalar, empirical Stieltjes transforms on the on-diagonals for which the sign is determined by Property \ref{prop:scalar_stieltjes_transform_monotonic_increasing}.
    This also proves item $(4)$ of Property \ref{prop:2conditional_stieltjes_transform_on_real_plane} since the scalar Stieltjes transforms on the on-diagonals are increasing for $s_1$ and $s_2$ increasing.
    All of the above holds as $N\uparrow\infty$, i.e. for $G_Z^{\mathcal D_2}$.
    \\
    Finally, we show that $G_Z^{\mathcal D_2}$ is bijective on $S_{pp}^Z$ for $p\in\{\pm\}$.
    Let $S\in\mathcal S_{pp}^Z$, and consider the region in which $\left(\mathfrak G_{Z_N}^{\mathcal D_2}\right)_{22}=g_2$ is constant.
    This region is a contour on $\mathcal S_{pp}^Z$ since increasing $s_1$ increases $\left(\mathfrak G_{Z_N}^{\mathcal D_2}\right)_{22}$, such that $s_2$ has to be decreased to return to the constant value.
    There can not be multiple solutions for the corresponding $s_2$ since the trace of the lower principal block in \eqref{eq:proof_prop_5_block_mat_inverse} is an empirical scalar Stieltjes transform, endowed with a functional inverse.
    In the following, we show that $\left(\mathfrak G_{Z_N}^{\mathcal D_2}\right)_{11}$ is strictly increasing with $s_1$ increasing for any contour $g_2=\mathrm{const}$ for all $N>N_0$, and conclude that $G_Z^{\mathcal D_2}$ is, therefore, bijective since  $G_Z^{\mathcal D_2}(S')=G_Z^{\mathcal D_2}(S'')$ if and only if $S'=S''$ on $\mathcal S_{pp}^Z$.
    The following derivations assume a generic $L\geq 2$.
    \\
    Consider the $L$-conditional Stieltjes transform as the map \begin{align}
        g: 
            (s_1,\ldots, s_L)
            \to
            \left[
                \left(\mathfrak G_{Z_N}^{\mathcal D_L}(S)\right)_{11},
                \ldots,
                \left(\mathfrak G_{Z_N}^{\mathcal D_L}(S)\right)_{LL}
            \right].
    \end{align}
    An infinitesimal change in the first-on diagonal of the $L$-conditional Stieltjes transform is given by
    \begin{align}
        \mathrm d \left(\mathfrak G_{Z_N}^{\mathcal D_L}(S)\right)_{11} 
        =
        \mathrm d g_1
        =
        \frac{\partial g_1 }{\partial s_1} \mathrm d s_1
        +
        \sum_{l\geq 2}
            \frac{\partial g_1 }{\partial s_l} \mathrm d s_l
        \label{eq:proof_prop_6_infinitesimal_change},
    \end{align}
    where $g_l$ denotes the $l^\mathrm{th}$ element of $g$.
    Let $[X]_{lk}$ denote block $(l,k)$ of the block matrix $X$, and $||X||_\mathrm{F}$ is the Frobenius norm of matrix $X$.
    The partial derivatives in \eqref{eq:proof_prop_6_infinitesimal_change} are elements of the Jacobian $J(S)$ of $g$, which is characterized by the following lemma.
    \begin{lemma}[Jacobian of $g$]
        \label{lem:jacboian_of_g}
        $J(S)=(J_{lk})_{l,k=1}^L$ has strictly positive elements and is positive definite on $\mathcal S_{\{p\}^L}^{Z}$ for $p\in\{\pm\}$, with 
        \begin{align}
            J_{lk} 
            =  
            \frac{1}{N}
            \left|\left|
                [(Z_N-S\otimes \mathbb I_N)^{-1}]_{lk}
            \right|\right|_\mathrm{F}^2.
        \end{align}
    \end{lemma}
    \begin{proof}[Proof of Lemma~\upshape{\ref{lem:jacboian_of_g}}]
        Define $A=(Z_N-S\otimes \mathbb I_N)^{-1}=(A_{lk})_{l,k=1}^L$ with $A_{lk}\in\mathbb R^{N\times N}$.
        The derivative of $A$ with respect to $s_k$ is
        \begin{align}
            \frac{\partial A}{\partial s_k}
            =
            A(e_ke_k^T\otimes \mathbb I_N)A,
            \label{eq:proof_prop_6_derivative_trace_arg_s_l}
        \end{align}
        where $e_k$ is the $k^\mathrm{th}$ standard basis vector.
        Block $(l,l)$ of \eqref{eq:proof_prop_6_derivative_trace_arg_s_l} is, hence, given by $A_{lk}A_{kl}$, such that
        \begin{align}
            J_{lk}
            =
            \mathrm{Tr}\left[
                \left[\frac{\partial A}{\partial s_k }\right]_{ll}
            \right]
            =
            \mathrm{Tr}\left[
                A_{lk}
                A_{kl}
            \right]
            =
            \frac{1}{N}
            \left|\left|
                [(Z_N-S\otimes \mathbb I_N)^{-1}]_{lk}
            \right|\right|_\mathrm{F}^2
            \geq0. 
        \end{align}
        It remains to show that $J\succ0$ on the domain $\mathcal S^Z_{\{p\}^L}$. 
        Consider the eigendecomposition $A=U\Lambda_AU^\dagger$ with $\lambda_{A,i}<0$ ($\lambda_{A,i}>0$) for all $i$ if $S\in\mathcal S_{\{-\}^L}^{Z}$ ($S\in\mathcal S_{\{+\}^L}^{Z}$), where $\lambda_{X,i}$ is the $i^\mathrm{th}$ eigenvalue of a square matrix $X$, and $U=[u_{1}, \ldots, u_{LN}]$ with $u_{i}\in\mathbb R^{LN}$.
        Partition every $u_i$ in vectors $u_{il}$ of size $N$ by $u_i=[u_{i1}^\dagger,\ldots,u_{iL}^\dagger]^\dagger$ with $l\in [L]$. 
        Block $(l,k)$ of A is then given by
        \begin{align}
            A_{lk} = \sum_{i=1}^{LN}
                \lambda_{A,i} u_{il}u_{ik}^\dagger.
        \end{align}
        Considering $\mathrm{tr}(A_{lk}A_{kl})=||A_{lk}||_F^2$, and by applying the cyclic property of the trace we obtain
        \begin{align}
            J_{lk}
            =
            \frac{1}{N}||A_{lk}||_F^2
            =
            \frac{1}{N}
            \sum_{i,j=1}^L 
                \lambda_{A,i} \lambda_{A,j}
                (u_{il}^\dagger u_{jl})^*
                (u_{ik}^\dagger u_{jk})
            = 
            \frac{1}{N}
            w_{l}^\dagger w_k,
            \label{eq:proof_prop_6_B_is_gram_matrix}
        \end{align}
        with $w_l = \mathrm{vect}\left(\left(\sqrt{\lambda_{A,i} \lambda_{A,j}}u_{il}^\dagger u_{jl}\right)_{i,j=1}^L\right)$, where $\mathrm{vect}(\cdot)$ vectorizes a matrix, and the square root returns a real value since A is a definite matrix.
        By \eqref{eq:proof_prop_6_B_is_gram_matrix} $J(S)$ is Gramian and hence necessarily positive semi-definite. 
        If all $w_l$ are linearly independent, $J(S)$ is positive definite. 
        We proof that this is fulfilled, by showing that the linear combination of all $w_l$ is zero if and only if all linear combining coefficients are zero.
        This is equivalent to the condition
        \begin{align}
            \sum_{l=1}^L \alpha_l u_{il}^\dagger u_{jl} 
            =
            u_i^\dagger (\mathrm{diag}\left\{\alpha_1,\ldots, \alpha_L\right\}\otimes \mathbb I_N) u_j
            \stackrel{?}{=}
            0, \text{for all }(i,j)\in [LN]^2.
            \label{eq:proof_prop_6_condition_linear_dependent}
        \end{align}
        Since $\{u_i\}$ is an orthonormal basis of $\mathbb R^{LN}$, \eqref{eq:proof_prop_6_condition_linear_dependent} can not be true unless $\alpha_l=0~\forall~l$ and, hence, the Jacobian is positive definite.
    \end{proof}
    The contour is defined by the configurations of $S$ on which the last $L-1$ elements of $g$, $g_{2:L} = [g_{2}, \ldots, g_{L}]$ are constant, i.e. 
    \begin{align}
        \mathrm d g_{2:L} 
        =
        J_{2:L,2:L} \mathrm d s_{2:L} + J_{2:L,1}\mathrm d s_1 
        = 0,
        \label{eq:proof_prop_6_change_on_contour_condition}
    \end{align}
    where $s_{2:L} = [s_2,\ldots, s_L]^T$, and  $J_{l:k,h:j} = (J(S))_{l:k,h:j}$ isolates the sub-matrix of $J(S)$ comprised of all elements with indices ranging from $(l,h)$ to $(k,j)$ for $l\leq k$ and $h\leq j$.
    By straightforward reformulation of \eqref{eq:proof_prop_6_change_on_contour_condition} and substitution in \eqref{eq:proof_prop_6_infinitesimal_change}, the derivative of $g_1$ on the contour is given as
    \begin{align}
        \left.
        \frac{
            \mathrm d \left(\mathfrak G_{Z_N}^{\mathcal D_L}(S)\right)_{11} 
        }{
            \mathrm d s_1
        }
        \right|_{\mathrm d g_{2:L}=0}
        =
        J_{11}
        - J_{1,2:L}J_{2:L,2:L}^{-1}J_{2:L,1},
        = 
        \frac{1}{(J(S)^{-1})_{11}}
        >0,
    \end{align}
    where the inequality follows from the fact that $J(S)$ is positive definite. 
    The same holds by symmetry for any other contour on which $L-1$ on-diagonal values of $\mathfrak G_{Z_N}^{\mathcal D_L}$ are constant. 
    
\end{proof}

\begin{proof}[Proof of Property~\upshape{\ref{prop:Lconditional_stieltjes_transform_on_real_plane}}]
    The proof follows by induction. Assume, that all items of Property~\ref{prop:Lconditional_stieltjes_transform_on_real_plane} hold for $L-1$.
    In order to derive the bounds, we analyze the $l^\mathrm{th}$ principal block of size $N\times N$ of $(Z_N-S\otimes \mathbb I_N)^{-1}$ for $N>N_0$, assuming that all required operator norms and minimum and maximum eigenvalues are finitely bounded for all $N>N_0$ and some $N_0>0$. 
    Let $R_l$ be the $L\times L$ block matrix which moves the $l^\mathrm{th}$ block row up to the first row such that
    \begin{align}
        R_l (Z_N-S\otimes \mathbb I_N) R_l^T 
        =
        \begin{bmatrix}
            E_{ll} - s_l\mathbb I_N & Q_l^T \\
            Q_l & Z_{\setminus l} - S_{\setminus l}\otimes\mathbb I_N
        \end{bmatrix}.
        \label{eq:proof_property_6_block_matrix_reordered}
    \end{align}
    The $l^\mathrm{th}$ principal block of $(Z_N-S\otimes \mathbb I_N)^{-1}$ is, hence, given by 
    \begin{align}
        \left[(Z_N-S\otimes \mathbb I_N)^{-1}\right]_{ll}        
        =
        \left[\left(R_l(Z_N-S\otimes \mathbb I_N)R_l^T\right)^{-1}\right]_{11},
        \label{eq:proof_L_stieltjes_i_block_inverse}
    \end{align}
   which evaluates to
    \begin{align}
        \left(R_l(Z_N-S\otimes \mathbb I_N)R_l^T\right)^{-1}
        =
        \begin{bmatrix}
            \left(
                \tilde E_{ll}
                - Q_l^T\tilde{Z}_{\setminus l}^{-1}Q_l
            \right)^{-1} 
            & \square_{N\times(L-1)N} \\ \square_{(L-1)N\times N} &
            \left(
                \tilde{Z}_{\setminus l} 
                - Q_l\tilde E_{ll}^{-1}Q_l^T
            \right)^{-1}
        \end{bmatrix},
        \label{eq:proof_prop_6_block_matrix_inverse_reorderd}
    \end{align}
    with $\tilde Z_{\setminus l} = Z_{\setminus l} - S_{\setminus l}\otimes \mathbb I_N$. 
    Assuming that the inverse of $\tilde Z_{\setminus l}$ exists and is negative definite, it is straightforward to show that the existence of \eqref{eq:proof_L_stieltjes_i_block_inverse} is guaranteed by 
    \begin{align}
        s_l 
        > 
        \lambda_\mathrm{max}(E_{ll}) 
        + 
        \left|\left|Q_l\right|\right|_{\infty}^2 
        \left|\left|\tilde Z_{\setminus l}^{-1}\right|\right|_{\infty}^2 
        \label{eq:proof_property_6_s_i_bound}.
    \end{align}
    Therefore, we have to show that the inverse exists. 
    The domain $\mathcal S_{\{-\}^L}^Z$ defines $S_{\setminus l} \in \mathcal S_{\{-\}^{L-1}}^{\hat Z^{(l)}}$, i.e. the matrix inverse of the lower principal block in \eqref{eq:proof_prop_6_block_matrix_inverse_reorderd} exists and is negative definite.
    Since, $-Q_l\tilde E_{ll}^{-1} Q_l^T$ is positive definite, this implies that $\tilde Z_{\setminus l}$ is negative definite and the matrix inverse exists. 
    With, \eqref{eq:proof_property_6_s_i_bound} we therefore also know that the upper principal block in \eqref{eq:proof_prop_6_block_matrix_inverse_reorderd} is negative definite, and so are all principal blocks in \eqref{eq:proof_property_6_block_matrix_reordered}.
    By Schur complement $Z_N-S\otimes \mathbb I_N$ is negative definite for all $S\in\mathcal{S}_{\{-\}^L}^Z$, the matrix inverse exists, and the corresponding (empirical) $L$-conditional Stieltjes transform is negative definite and increasing, towards zero, with $S$ increasing.\\
    It remains to show that the functional inverse exists for $S\in\mathcal S_{\{-\}^L}^Z$.
    To this end we consider the case in which $L-1$ on-diagonal elements of the $L$-conditional Stieltjes transform are constant, as in the proof of Property~\ref{prop:2conditional_stieltjes_transform_on_real_plane}.
    Assume, without loss of generality, that the first on-diagonal element varies and the remaining $L-1$ elements are kept constant, as 
    \begin{align}
        \left(
            \mathfrak G_{Z_N}^{\mathcal D_L}(S)
        \right)_{2:L,2:L}
        =
        \mathfrak G_{\hat Z^{(1)}}^{\mathcal D_{L-1}}\left(S_{\setminus 1}\right)
        = \mathrm {diag}\{g_{2:L}\}
        =
        \mathrm{const}.
        \label{eq:proof_prop_6_constant_value_bijective}
    \end{align} 
    The second equality follows from \eqref{eq:proof_prop_6_block_matrix_inverse_reorderd} and the definition of $\hat Z^{(1)}$ in Property~\ref{prop:Lconditional_stieltjes_transform_on_real_plane}.
    The functional inverse of this Stieltjes transform exists for all $S_{\setminus 1}\in \mathcal S_{\{-\}^{L-1}}^{\hat Z^{(1)}}$ and large enough N, i.e. the value of $S_{\setminus 1}$ at which \eqref{eq:proof_prop_6_constant_value_bijective} is constant is unique for every $s_1$ and this holds for all $S \in \mathcal S_{\{-\}^L}^Z$ by definition.
    Therefore, the region in which \eqref{eq:proof_prop_6_constant_value_bijective} is constant admits a contour, on which $g_1$ is strictly increasing by the same argument as in the Proof of Property~\ref{prop:2conditional_stieltjes_transform_on_real_plane}.
    
\end{proof}
% The following proofs are based on the approach by Gaussian representation of the Haar measure, see \cite{guionnetFourierViewTransform2005} for an introduction to the method.

\begin{proof}[Proof of Theorem~\upshape{\ref{thm:pL_rank1_integral}}]
Let $M(N)=1$ and $\beta=1$, then \eqref{eq:L_coupled_spherical_integrals} reduces to
\begin{align}
    I_N^{(1)}(\{D_l\},\{E_{lk}\})=
    \int 
        \exp\left\{
            \sum_{l,k=1}^L
            N
            \sqrt{\theta_l \theta_k} u_{l,1}^T E_{lk}u_{k,1}
        \right\}
        \prod_{l=1}^L
            \mathrm d m_N^{(1)}(U_l)   
        \label{eq:integral_L_rank1},
\end{align}
where $\theta_l$ is the single non-zero eigenvalue of $D_lD_l^T$, and $u_{l,i}\in\mathbb R^{N\times 1}$ is the $i^\mathrm{th}$ column of $U_l=[u_{l,1},\ldots, u_{l,N}]$.
As in \cite{guionnetFourierViewTransform2005}, we replace all Haar measure by the multivariate Gaussian measure $\mu$ of zero mean and identity covariance matrix, and substitute $u_{l,1}$ by $g_{l,1}/||g_{l,1}||_2$, with $||\cdot||_2$ representing the $L2$-norm, to obtain
\begin{align}
    I_N^{(1)}(\{D_l\},\{E_{lk}\})
    =
    \int
        \exp\left\{
            \sum_{l,k=1}^L
            N\sqrt{\theta_l \theta_k} \frac{g_{l,1}^T}{||g_{l,1}||_2}E_{lk}\frac{g_{k,1}}{||g_{k,1}||_2}
        \right\}
        \prod_{l=1}^L
            \mathrm d \mu(g_{l,1}).
\end{align}
Consider the event
\begin{align}
    \mathcal A_N(\kappa) 
    = 
    \left\{
        \left|
            \frac{||g_{l,1}||_2^2}{N}
            -1
        \right|
        \leq N^{-\kappa}
        ~\forall~ l\in[L]
    \right\}
    \label{eq:event_rank1_P_2}.
\end{align}
We note that this event is similar to event $\mathcal B_N(\kappa)$ of \cite{guionnetFourierViewTransform2005} in the proof of Theorem 7, referred to as $\mathcal B_N^\mathrm{GM}(\kappa)$, henceforth.
The difference lies in the fact that in our case the Gaussian vectors do not have to be asymptotically orthogonal as the column vectors of $U_l$ and $U_k$ for $l\neq k$ are not necessarily orthogonal.
Nonetheless, the following bounds from \cite{guionnetFourierViewTransform2005} still hold
\begin{align}
    \mathbb E\left[\mathfrak 1_{\mathcal A_N(\kappa)} e^{NF}\right]
    \leq 
    I_N^{(1)}(\{D_l\},\{E_{lk}\})
    \leq 
    (1+\epsilon(N,\kappa))\mathbb E\left[\mathfrak 1_{\mathcal A_N(\kappa)}e^{NF}\right]
    \label{eq:p_2_rank_1_general_bounds},
\end{align}
with $\epsilon(N,\kappa)$ going to zero as $N\uparrow\infty$ if $\kappa<1/2$ and $\mathfrak 1$ being the indicator function.
The exponent term $F$ is given by
\begin{align}
    F 
    =
    \sum_{l=1}^L
    \underbrace{
        \theta_l \frac{g_{l,1}^T}{||g_{l,1}||_2} E_{ll}\frac{g_{l,1}}{||g_{l,1}||_2}
        }_{F_{ll}}
        +
        \sum_{k\neq l}
            \underbrace{
            \sqrt{\theta_l \theta_k} \frac{g_{l,1}^T}{||g_{l,1}||_2} E_{lk}\frac{g_{k,1}}{||g_{k,1}||_2}
            }_{F_{lk}}
    \label{eq:P2_rank1_F_exponent_terms}.
\end{align}
We will now bound each summand of \eqref{eq:P2_rank1_F_exponent_terms} in order to remove the denominator containing $||g_{l,1}||_2$ to facilitate Gaussian integration.
In the following, the lower bound to \eqref{eq:p_2_rank_1_general_bounds} is derived explicitly, the upper bound follows analogously.
The $F_{ll}$ exponent terms are structurally similar to the formulation in \cite[Equation ($16$)]{guionnetFourierViewTransform2005}, such that
\begin{align}
    \exp\left\{
        NF_{ll} 
    \right\}
    \geq 
    \exp\left\{
        N\theta_lv_l 
        - 
        N^{1-\kappa}
        \theta_l(||E_{ll}||_\infty + |v_l|)
    \right\}    
    \exp\left\{
        \theta_lg_{l,1}^T (E_{ll} - v_l\mathbb I_N)g_{l,1}
    \right\}
    \label{eq:exp_Q_ll_lower_bound},
\end{align}
where $v_l$ is introduced to control the value around which the integral result concentrates as $N\uparrow\infty$, since the $N\theta_l v_l$ terms will dominate the exponent. 
Therefore, $v_l$ is referred to as a concentration variable in the following.
The computations for the exponent terms $F_{lk}$ with $l\neq k$ are repeated explicitly, as
\begin{align}
    NF_{lk}
    =~&
        N\sqrt{\theta_l \theta_k}\frac{g_{l,1}^TE_{lk}g_{k,1}}{||g_{l,1}||_2 ||g_{k,1}||_2} 
        + \sqrt{\theta_l \theta_k}g_{l,1}^TE_{lk}g_{k,1}
        - \sqrt{\theta_l \theta_k}g_{l,1}^TE_{lk}g_{k,1}
    \label{eq:Flk_rank1A}
    \\
    \geq~&
    \sqrt{\theta_l \theta_k}g_{l,1}^TE_{lk}g_{k,1}
    -
    \sqrt{\theta_l\theta_k}
    \left|
        \frac{N}{||g_{l,1}||_2 ||g_{k,1}||_2}-1
    \right|
    \left|
        g_{l,1}^TE_{lk}g_{k,1}
    \right|
    \\
    \geq~&
    \sqrt{\theta_l \theta_k}g_{l,1}^TE_{lk}g_{k,1}
    -
    N^{1-\kappa}
    \sqrt{\theta_l\theta_k}
    ||E_{l,k}||_\infty
    \label{eq:Flk_lower_bound}.
\end{align}
Bound \eqref{eq:Flk_lower_bound} is obtained by noting that $|g_{l,1}^TE_{lk}g_{k,1}|\leq ||g_{l,1}||_2||g_{k,1}||_2 ||E_{lk}||_\infty$ and by the following bounds on the squared $L2$-norm based on \eqref{eq:event_rank1_P_2}
\begin{align}
    N(1-N^{-\kappa})\leq ||g_{l,1}||_2^2\leq N(1+N^{-\kappa})
    \label{eq:bound_squared_l2_gaussian_vect}.
\end{align}
For sufficiently large $N$, the lower bound is positive, such that the $L2$-norm is bounded by the corresponding square roots
\begin{align}
    \sqrt{N(1-N^{-\kappa})}\leq ||g_{l,1}||_2\leq \sqrt{N(1+N^{-\kappa})}.
\end{align}
Note that, unlike in \eqref{eq:exp_Q_ll_lower_bound}, no additional concentration variable is introduced in the lower bound of $NF_{lk}$. 
There are $L$ integration variables such that we only require $L$ variables $v_l$ to control the asymptotic concentration of the integrals.
The coupling of the integration variables causes a coupling of the concentration variables $v_l$.
Applying \eqref{eq:exp_Q_ll_lower_bound} and \eqref{eq:Flk_lower_bound} to \eqref{eq:p_2_rank_1_general_bounds} yields the lower bound
\begin{align}
    I_N^{(1)}(\{D_l\},\{E_{lk}\})
    \geq~&
    \exp\left\{
        N\sum_{l=1}^L\theta_l v_l
    \right\}
    \exp\left\{
        -N^{1-\kappa}\sum_{l=1}^L
            \theta_l(||E_{ll}||_\infty+|v_l|)+\sum_{k\neq l}\sqrt{\theta_l \theta_k} ||E_{lk}||_\infty
    \right\}
    \\
    &\times
    \underbrace{
        \mathbb E\left[
            \mathfrak 1_{\mathcal A_N(\kappa)}
            \exp\left\{
                \sum_{l=1}^L
                    \theta_l(g_{l,1}^TE_{ll}g_{l,1} - v_l g_{l,1}^Tg_{l,1})
                    +
                    \sum_{k\neq l}
                        \sqrt{\theta_l\theta_k}g_{l,1}^TE_{lk}g_{k,1}
            \right\}
        \right]
    }_{\Xi_\mathrm{lb}^{(1)}}.
    \nonumber
\end{align}
To simplify the notation we introduce the shorthand $\Xi_\mathrm{lb}^{(1)}$, which signifies the expectation with respect to the Gaussian measures in the lower bound computation for $M(N)=1$.
It is computed by means of Gaussian integration.
Averaging over all $g_{l,1}$ separately is tedious as the Gaussian vectors are coupled through the $F_{lk}$ terms. 
Instead, we define the joint distribution $g_1=[g_{l,1}^T,\ldots, g_{L,1}^T]^T \sim \mathcal N(0, \mathbb I_{LN})$ and compute the corresponding joint average, where $\mathcal N(m, C)$ is the multivariate real Gaussian distribution with mean $m$ and covariance matrix $C$.
The Gaussian integral reduces to
\begin{align}
    \Xi_\mathrm{lb}^{(1)}
    =~&
    \int
        \mathfrak 1_{\mathcal A_N(\kappa)}
        \frac{1}{(2\pi)^{LN/2}}
        \exp\left\{
            -\frac{1}{2}
            g_1^T
            \Lambda
            g_1
        \right\}
        \mathrm d g _1
        \label{eq:expecation_lb_2_1},
\end{align}
with
\begin{align}
    \Lambda
    =
    \begin{bmatrix}
        -2\theta_1 E_{11} + (1+2\theta_1v_1)\mathbb I_N & -2\sqrt{\theta_1\theta_2}E_{12} & \ldots 
        \\
        -2\sqrt{\theta_1\theta_2}E_{21} & -2\theta_2 E_{22}+ (1+2\theta_2v_2)\mathbb I_N  \\
        \vdots && \ddots \\
    \end{bmatrix}
    \label{eq:lambda_definition_rank_1}.
\end{align}
Recall that $E_{12}=E_{21}^T$ and $\theta_l\in \mathbb R^+$. 
The block matrix \eqref{eq:lambda_definition_rank_1} is interpreted as the precision matrix of the Gaussian vector $g_1\sim \mathcal{N}\left(0,\Lambda^{-1}\right)$.
The expectation is hence given by
\begin{align}
    \Xi_\mathrm{lb}^{(1)}
    =
    |\Lambda|^{-\frac{1}{2}}
    \int
        \mathfrak 1_{\mathcal A_N(\kappa)}
        \mathrm d \mu(g_1; 0, \Lambda^{-1})
    =
    |\Lambda|^{-\frac{1}{2}}
    P_N(\mathcal A_N(\kappa))
    \label{eq:P2_rank1_Xi}.
\end{align}
Where $\mu(\cdot;m,C)$ is the real, multivariate Gaussian measure of mean $m$ and covariance matrix $C$.
The task is hence to bound the probability of event $\mathcal A_N(\kappa)$ under a Gaussian probability measure with zero mean and variance $\Lambda^{-1}$.
It is lower bounded by
\begin{align}
    P_N(\mathcal A_N(\kappa))
    =~&
    \mathrm{Pr}\left\{
        \left|
            \frac{||g_{l,1}||_2^2}{N}-1
        \right|
        \leq 
        N^{-\kappa}
        ~\forall~l
    \right\}
    =
    1-\mathrm{Pr}\left\{
        \left|
            \frac{||g_{l,1}||_2^2}{N}-1
        \right|
        > 
        N^{-\kappa}
        ~\forall~l
    \right\}
    \\
    \geq~&
    1
    -
    \sum_{l=1}^L
        \mathrm{Pr}\left\{
            \left|
                \frac{||g_{l,1}||_2^2}{N}-1
            \right|
            > 
            N^{-\kappa}
        \right\}
    \label{eq:P2_rank_1_lower_bound_pevent_probability}.
\end{align}
In order to bound the probabilities in \eqref{eq:P2_rank_1_lower_bound_pevent_probability} by Chebyshev's inequality, consider $||g_{l,1}||_2^2/N$ as an RV.
For the bound to be applicable the RV must have unit mean.
The distribution of the RVs is, however, coupled as $\Lambda^{-1}$ is not block-diagonal.
Therefore, we jointly enforce the condition for all $l\in[L]$ by
\begin{align}
    \mathrm{Tr}_L\left[
        \mathbb E\left[
            g_1g_1^T
        \right]
    \right]
    =\mathrm{Tr}_L\left[
        \Lambda^{-1}
    \right]
    =\mathbb I_L.
    \label{eq:P2_rank1_unit_mean_constraint}
\end{align}
Normalization by $1/N$ is done implicitly by the normalized-partial trace $\mathrm{Tr}_L$.
The $L$ scalar variables $v_l$ will be chosen such that \eqref{eq:P2_rank1_unit_mean_constraint} is fulfilled. 
This implies that $L^2$ variables are controlled by $L$ degrees of freedom.
Therefore, we must introduce some additional constraint which reduces this condition to $\{v_l\}$ controlling $L$ variables only.
The next steps reveal that the trace in \eqref{eq:P2_rank1_unit_mean_constraint} is an empirical $L$-conditional Stieltjes transform. 
By assuming that the corresponding R-transform is diagonal, we force the off-diagonals of the Stieltjes transform to be zero.
This is reflected in the condition that block matrix $Z$ is R-cyclic.
Define $\hat{S}$ and $\hat Z_{N}$ as 
\begin{align}
    \hat{S} = 
    -P^{\frac{1}{2}}SP^{\frac{1}{2}}
    =
    -P^{\frac{1}{2}}
    \mathrm{diag}\left\{
        \frac{1}{2\theta_1}+ v_1, \ldots, \frac{1}{2\theta_L} +v_L
    \right\}
    P^{\frac{1}{2}}
    \label{eq:L_rank1_definition_Shat}
\end{align}
and
\begin{align}
    \hat{Z}_N
    = 
    -\begin{bmatrix}
        2\theta_1 E_{11} &\ldots & 2\sqrt{\theta_1\theta_L} E_{1L}
        \\
        \vdots & \ddots & \vdots
        \\
        2\sqrt{\theta_1 \theta_L} E_{L1} &\ldots& 2\theta_L E_{LL}
    \end{bmatrix}
    =
    -
    P^\frac{1}{2}Z_NP^\frac{1}{2}
    \label{eq:L_rank1_definition_Zhat},
\end{align}
recall $P = \mathrm{diag}\left\{2\theta_1,\ldots,2\theta_L\right\}$ and $Z_N$ as in Definition~\ref{def:operator_valued_convergence}.
Then, \eqref{eq:P2_rank1_unit_mean_constraint} is found to introduce the empirical operator-valued Stieltjes transform as
\begin{align}
    \mathrm{Tr_L}\left[
        \mathbb E\left[
            g_1g_1^T
        \right]
    \right]
    =
    \mathrm{Tr_L}\left[
        \left(\hat Z_N- \hat{S}\otimes \mathbb I_{N}\right)^{-1}
    \right]
    =
    \mathfrak{G}_{\hat Z_N}^{\mathcal D_L}\left(\hat{S}\right)
    =
    \mathbb I_L.
    \label{eq:P2_rank1_OVStilTrans_condition}
\end{align}
Note that due to the R-cyclic assumption the empirical operator-valued Stieltjes transform on $\mathcal B_L$ collapses on $\mathcal D_L$ for all $\hat S\in \mathcal D_L$.
Next, $v_l$ shall be determined such that \eqref{eq:P2_rank1_OVStilTrans_condition} is fulfilled. 
The variables of interest are located in the argument of the Stieltjes transform. 
Applying the scaling Properties \ref{prop:operator_valued_stieltjes_transform_left_scaling} and \ref{prop:operator_valued_stieltjes_transform_right_scaling} followed by some algebraic manipulation and taking the functional inverse yields
\begin{align}
    \mathfrak{K}_{Z_N}^{\mathcal D_L}(-P) = S = \mathrm{diag}\{v_1,\ldots, v_L\} + P^{-1}.
\end{align}
We solve for $v_l$ by reformulation, which are given by the empirical operator-valued R-transform of $Z$ evaluated at $P$, as
\begin{align}
    \mathrm{diag}\{v_1, \ldots,v_L\}
    =
    \mathfrak{R}_{Z_N}^{\mathcal{D}_L}(P) = \mathfrak{K}_{Z_N}^{\mathcal D_L}(-P)-P^{-1}.
\end{align}
Chebyshev's inequality, furthermore, requires us to compute the variance of $||g_{l,1}||_2^2$.
To remove the correlation of the entries of $g_{l,1}$ we left multiply $g_{l,1}$ by $U_{\Lambda,l,1}^T$ which diagonalizes the $N\times N $ block of $\Lambda^{-1}$ at index $(l,l)$ upon left and transposed right multiplication.
Then, by substituting $h_{l,1} =U_{\Lambda,l,1}^T g_{l,1}$, all Gaussian vectors in \eqref{eq:P2_rank1_diagonalized_gaussian_vectors} have independent entries.
Applying \eqref{eq:P2_rank1_unit_mean_constraint} further simplifies the squared mean to
\begin{align}
    \sigma_{l,1}^2 
    &= 
    \frac{1}{N^2}\mathbb E\left[h_{l,1}^Th_{l,1}h_{l,1}^Th_{l,1}\right]
    -\frac{1}{N^2} \mathbb E\left[h_{l,1}^Th_{l,1}\right]^2
    \label{eq:P2_rank1_diagonalized_gaussian_vectors}
    \\
    &=
    \frac{1}{N^2}\sum_{\substack{i,j=1\\ i\neq j}}^N
        \overline{h_{l,1,i}^2}
        ~
        \overline{h_{l,1,j}^2}
    +\frac{1}{N^2} \sum_{i=1}^N\overline{h_{l,1,i}^4}
    -1
    \label{eq:P2_rank1_diagonalized_gaussian_vectors_simp1}
    \\
    &=
    \frac{1}{N^2}\sum_{i,j=1}^N
        \overline{h_{l,1,i}^2}
        ~
        \overline{h_{l,1,j}^2}
    +\frac{2}{N^2} \sum_{i=1}^N\overline{h_{l,1,i}^2}^2
    -1
    \label{eq:P2_rank1_diagonalized_gaussian_vectors_simp2}
    \\
    &=
    \left(\left(\mathrm{Tr}_L\left[\Lambda^{-1}\right]\right)_{ll}\right)^2
    +\frac{2}{N^2} \sum_{i=1}^N \lambda_{[\Lambda^{-1}]_{ll},i}^2
    -1
    \label{eq:P2_rank1_diagonalized_gaussian_vectors_simp3}
    \\
    &=
    \frac{2}{N} \mathrm{Tr}\left[\left[\Lambda^{-1}\right]_{ll}^2\right]
    \label{eq:P2_rank1_diagonalized_gaussian_vectors_simp4},
\end{align}
where $h_{l,1,i}$ is the $i^\mathrm{th}$ element of vector $h_{l,1}$.
To bound \eqref{eq:P2_rank1_diagonalized_gaussian_vectors_simp4}, we compute the $l^\mathrm{th}$ principal block  of $\Lambda^{-1}$ by
\begin{align}
    \Lambda^{-1}
    =
    -
    P^{-\frac{1}{2}}
    R_l^T
    \begin{bmatrix}
        (\tilde{E}_{ll}-Q_l^T\tilde Z_{\setminus l}^{-1}Q_l)^{-1} & \square_{N\times (L-1)N}
        \\
        \square_{(L-1)N\times N} & \square_{(L-1)N\times(L-1)N}
    \end{bmatrix}
    R_l
    P^{-\frac{1}{2}},
    \label{eq:P2_rank1_invers_reformulation}
\end{align}
assuming that all required inverses exist. 
In fact, this assumption must be correct as $\mathfrak{G}^{\mathcal D_L}_{Z_N}(S)$ does exist. 
Recall that $\tilde E_{ll} = E_{ll}-s_l\mathbb I_N$ and similarly $\tilde Z_{\setminus l}=Z_{\setminus l}-S_{\setminus l}\otimes \mathbb I_N$.
Let $\tilde{C}_{ll} = \tilde E_{ll}-Q_l^T\tilde Z_{\setminus l}^{-1} Q_l$, then \eqref{eq:P2_rank1_diagonalized_gaussian_vectors_simp4} is bounded by
\begin{align}
    \sigma_{l,1}^2 
    &=
    \frac{2}{N^2} 
    \frac{1}{(2\theta_l)^2}
    \sum_{i=1}^N
        \frac{1}{\lambda_{\tilde{C}_{ll},i}^2}
    \leq
    \frac{2}{N}\frac{1}{\left(2\theta_l \lambda_\mathrm{max}\left(\tilde{C}_{ll}\right)\right)^2}
    \label{eq:P2_rank1_variance_ll_bound_1},
\end{align}
where we maximize over the eigenvalues of $\tilde C_{ll}$ prior to taking the power of two, as $\tilde C_{ll}$ is negative definite.
Since the entries of $P^{1/2}$ are singular values of $D_l$, $P^{-1/2}\succ 0$ is guaranteed. 
Hence, for $\Lambda^{-1}\succ 0$, which is necessary as $\Lambda^{-1}$ is a covariance matrix, we require $\tilde C_{ll}\prec 0$. 
To lower bound the absolute value of the maximum eigenvalue of $\tilde C_{ll}$, $\mathrm{Tr}\left[\tilde C_{ll}^{-1}\right]$ is interpreted as a scalar Stieltjes transform.
This of course implies that $\mathrm{Tr}\left[\tilde C_{ll}^{-1}\right]<0$.
The scalar Stieltjes transform at $v_l+1/(2\theta_l)$ is given as
\begin{align}
    \mathrm{Tr}\left[
        \tilde C_{ll}^{-1}
    \right]
    =
    \mathrm{Tr}\left[
        \left(
            \smash{\underbrace{
                E_{ll}
                -
                Q_l^T
                \tilde Z_{\setminus l}^{-1}
                Q_l
            }_{C_{ll}}}
            - \left(v_l+\frac{1}{2\theta_l}\right)\mathbb I_N
        \right)^{-1} 
    \right]
    = \mathfrak{G}_{C_{ll}}\left(v_l+\frac{1}{2\theta_l}\right).
\end{align}
\vspace{10pt}\\
Based on Property~\ref{prop:scalar_stieltjes_transform_monotonic_increasing} we know that $\mathfrak{G}_{C_{ll}}(s)<0~\forall s\geq \lambda_\mathrm{max}(C_{ll}) +\eta$ for some $\eta>0$, such that
\begin{align}
    |\eta| \leq \left|v_l+\frac{1}{2\theta_l} - \lambda_\mathrm{max}(C_{ll})\right| = \left|\lambda_\mathrm{max}(\tilde C_{ll})\right|,
    \label{eq:P2_rank1_eta_bound}
\end{align}
which allows us to bound the variance by
\begin{align}
    \sigma_{l,1}^2 
    \leq 
    \frac{1}{
        2N\theta_l^2\eta^2
    }.
    \label{eq:sigmaBoundRank1}
\end{align}
Chebyshev's inequality is applied to the converse probabilities of \eqref{eq:P2_rank_1_lower_bound_pevent_probability} and bounded by \eqref{eq:sigmaBoundRank1} to obtain
\begin{align}
    \mathrm{Pr}\left\{
        \left|
            \frac{||g_{l,1}||_2^2}{N}-1
        \right|
        > 
        N^{-\kappa}
    \right\}
    \leq N^{2\kappa-1} (2\theta_l^2\eta^2)^{-1}
    \forall l.
\end{align}
If $\kappa<1/2$, there exists an $N>0$ for which $P_N(\mathcal A_N(\kappa))\geq1/2$.
We conclude that \eqref{eq:integral_L_rank1} is lower bounded by
\begin{align}
    I_N^{(1)}(\{D_l\},\{E_{lk}\})
    \geq~&
    \frac{1}{2}
    |\Lambda|^{-\frac{1}{2}}
    \exp\left\{
        N\sum_{l=1}^L\theta_l v_l
    \right\}
    \exp\left\{
        -N^{1-\kappa}\sum_{l,k=1}^L
            \sqrt{\theta_l\theta_k}||E_{lk}||_\infty
        -N^{1-\kappa}\sum_{l=1}^L
            |\theta_l||v_l|
    \right\}.
\end{align}
The upper bound in \eqref{eq:p_2_rank_1_general_bounds} is obtained in similar fashion to the lower bound.
Upper bounding the exponent terms of $F_{lk}$ in \eqref{eq:P2_rank1_F_exponent_terms} is done by inverting the sign of the terms scaling with $N^{1-\kappa}$, and the probability $P_N(\mathcal A_N(\kappa))$ is upper bounded by $1$ to obtain
\begin{align}
    I_N^{(1)}(\{D_l\},\{E_{lk}\})
    \leq~&
    (1+\epsilon(N,\kappa))
    |\Lambda|^{-\frac{1}{2}}
    \exp\left\{
        N\sum_{l=1}^L\theta_l v_l
    \right\}
    \exp\left\{
        N^{1-\kappa}\sum_{l,k=1}^L
            \sqrt{\theta_l \theta_k} ||E_{lk}||_\infty
        +
        N^{1-\kappa}\sum_{l=1}^L
            |\theta_l|v_l| 
    \right\}
    \label{eq:upper_bound_rank_1}.
\end{align}
Following the approach of \cite{guionnetFourierViewTransform2005} we can now compute the normalized logarithm of the integral and the corresponding bounds to attain
\begin{align}
        &-\frac{1}{N}\log 2 
        -
        N^{-\kappa} \sum_{l,k=1}^L
            \sqrt{\theta_l \theta_{k}} ||E_{l k}||_\infty
        -
        N^{-\kappa} \sum_{l=1}^L
            |\theta_l||v_l|
        \label{eq:P2_rank1_logarithmic_bound_reformulation}
        \\
        &\leq \frac{1}{N}\log I_N^{(1)}(\{D_l\},\{E_{lk}\})-\sum_{l=1}^{L}\theta_lv_l + \frac{1}{2N}\log |\Lambda|
        \nonumber
        \\
        &\leq \frac{1}{N}\log (1+\epsilon(N,\kappa))
        +
        N^{-\kappa} \sum_{l,k=1}^L
            \sqrt{\theta_l \theta_{k}} ||E_{lk}||_\infty
        +
        N^{-\kappa} \sum_{l=1}^L 
            |\theta_l||v_l|
        \nonumber.
\end{align}
The upper and lower bounds tend to zero as $N\uparrow\infty$.
Hence, the following asymptotic equivalence must be satisfied
\begin{align}
    \lim_{N\uparrow\infty}
    \frac{1}{N}\log I_N^{(1)}(\{D_l\},\{E_{lk}\})
    % =
    % \lim_{N\uparrow\infty}
    % - \frac{1}{2N}\log |\Lambda|
    % +
    % \sum_{l=1}^{L}\theta_lv_l 
    =
    \lim_{N\uparrow\infty}
    - \frac{1}{2N}\log |\Lambda|
    +
    \frac{1}{2}\mathrm{tr}\left(
        P \mathfrak{R}_{Z_N}^{\mathcal D_L}(P)
    \right)
    = g(\theta_1, \ldots, \theta_L)
    \label{eq:P2_rank1_asymptotic_eq_integral_after_bounding},
\end{align}
with $\sum_l \theta_lv_l = \frac{1}{2}\mathrm{tr}\left(P \mathfrak R_{Z_N}^{\mathcal D_L}(P)\right)$.
The determinant term is reformulated as
\begin{align}
    \frac{1}{2N}\log|\Lambda|
    =
    \frac{1}{2}\log |P|
    +\frac{1}{2N} \log \left|
        \left(\mathfrak{R}_{Z_N}^{\mathcal D_L}(P)+P^{-1}\right)\otimes \mathbb I_N-Z_N
    \right|.
\end{align}
Note, that $g(0,\ldots,0)=0$ since $\Lambda=\mathbb I_{LN}$ at $\theta_l=0~\forall l$. 
Therefore, $g(\theta_1,\ldots,\theta_L)$ can be reformulated by derivation and subsequent integration over the path $r(w) = [\theta_1 w,\ldots, \theta_Lw]$ with $0\leq w\leq 1$. 
By applying the property $1/N \mathrm{tr} = \mathrm{tr} \circ \mathrm{Tr_L}$ on an $LN\times LN$ matrix and considering the definition of the $L$-conditional Stieltjes and R-transform with the corresponding convergence properties we attain
\begin{align}
    \nabla g(\theta_1w,\ldots,\theta_Lw) =
    \begin{bmatrix}
        [R_Z(Pw)]_{11}\\
        \vdots
        \\
        [R_Z(Pw)]_{LL}
    \end{bmatrix}.
\end{align}
An equivalent form of $g(\theta_1,\ldots,\theta_L)$ is recovered by integration over the path $r(w)$ to obtain
\begin{align}
    g(\theta_1,\ldots,\theta_L)
    &= \int_{0}^1
        \nabla g(\theta_1w,\ldots,\theta_Lw)^T r'(w) 
        \mathrm d w
    = \frac{1}{2}\int_{0}^1 
        \begin{bmatrix}
            [R_Z(Pw)]_{11}
            \\
            \vdots
            \\
            [R_Z(Pw)]_{LL}
        \end{bmatrix}^T
        \begin{bmatrix}
            2 \theta_1 \\
            \vdots \\
            2 \theta_L
        \end{bmatrix}
        \mathrm d w=
    \frac{1}{2}
    \int_0^1
        \mathrm{tr}\left(
            R_{Z}(Pw)P
        \right)
        \mathrm d w
    \label{eq:g_func_definition_rank_1}.
\end{align}
In the case of $\beta=2$, replace the real Gaussian measure in \eqref{eq:expecation_lb_2_1} by its complex-valued counterpart and adjust the scaling of the exponent. 
Hence, $\theta_l$ is divided by $2$ in every subsequent step, but on the path $r(w)$ which remains unchanged. 
\end{proof}

\begin{proof}[Proof of Theorem~\upshape{\ref{thm:pL_rankM_integral}}]
Let $\beta=1$ and $M(N)>1$, then \eqref{eq:L_coupled_spherical_integrals} is given as
\begin{align}
    I_N^{(1)}(\{D_l\},\{E_{lk}\})
    =
    \int
        \exp\left\{
            \sum_{i=1}^{M(N)}
            \sum_{l,k=1}^L
            N\sqrt{\theta_{k,i}\theta_{l,i}} u_{l,i}^T E_{lk}u_{k,i}
        \right\}
        \prod_{l=1}^L
            \mathrm d m_N^{(1)}(U_l)
        \label{eq:integral_P2_rankM},
\end{align}
where $\sqrt{\theta_{l,i}}$ is the $i^\mathrm{th}$ singular value of $D_l$. 
The column vectors $u_{l,i}$ and $u_{l,j}$ are orthogonal for all $i\neq j$. 
We can, therefore, no longer replace them by normalized Gaussian vectors.
Instead, following the approach of \cite[Section 2.2]{guionnetFourierViewTransform2005}, we substitute with the corresponding Schmidt orthogonalized Gaussian $\tilde{g}_{l,i}$, leading to
\begin{align}
    \label{eq:integral_P2_rankM_orthogonalized}
    I_N^{(1)}(\{D_l\},\{E_{lk}\})
    =
    \int
        \exp\left\{
            \sum_{i=1}^{M(N)}
            \sum_{l,k=1}^L
            N\sqrt{\theta_{l,i} \theta_{k,i}} \frac{\tilde g_{l,i}^T}{||\tilde{g}_{l,i}||_2}E_{lk}\frac{\tilde g_{k,i}}{||\tilde{g}_{k,i}||_2}
        \right\}
        \prod_{l=1}^L
            \mathrm d \mu(g_{l,i}).
\end{align}
Note that the Gaussian vectors $g_{l,i}$ are orthogonalized over index $i$ but not over index $l$, as $u_{l,i}$ and $u_{k,j}$ are columns of independent Haar-distributed matrices of the orthogonal group if $k\neq l$.
Consider the event
\begin{align}
    \mathcal B_N(\kappa)=
    \left\{
        \left|
            \frac{||g_{l,i}||_2^2}{N}
            -1
        \right|
        \leq N^{-\kappa}
        ~\forall~ (l,i)\in \mathcal V
        ,
        \left|
            \frac{g_{l,i}^Tg_{l,j}}{N}
        \right|
        \leq N^{-\kappa}
        ~\forall~ (l,i)\neq (l,j) \in \mathcal V
    \right\}
    \label{eq:event_rankM_P2},
\end{align}
with $\mathcal V = \{(l,i)|l\in[L],i\in[M(N)]\}$.
Consider the complement event, $\mathcal B_N^\mathrm c(\kappa)$. 
The Gaussian are \textit{i.i.d.}, such that the complement event probability is bounded by
\begin{align}
    \mathrm{Pr}\left\{\mathcal B_N^\mathrm{c}(\kappa)\right\}
    &\leq
    \sum_{(l,i)\in \mathcal{V}}
        \mathrm{Pr}\left\{
            \left|
                \frac{||g_{l,i}||_2^2}{N}
                -1
            \right|
            > N^{-\kappa}
        \right\}
    +
    \sum_{\substack{(l,i),(k,j)\in \mathcal{V}\\(l,i)\neq(k,j)}}
        \mathrm{Pr}\left\{
            \left|
                \frac{g_{l,i}^Tg_{k,j}}{N}
            \right|
            > N^{-\kappa}
        \right\}
        \leq 2 C'(\kappa)e^{-\alpha N^{1-2\kappa}},
\end{align}
with finite-valued $C'(\kappa)$, for any $0<\kappa<1/2$ and large enough $N$, a corresponding bounding $\alpha>0$ exists as shown in \cite[Section 2.2]{guionnetFourierViewTransform2005}.
Therefore, the $L$-fold coupled spherical integral with $M(N)>1$ is lower and upper bounded with $\epsilon(N,\kappa)$ going to zero if $\kappa<1/2$ by
\begin{align}
    \mathbb E\left[\mathfrak 1_{\mathcal B_N(\kappa)} \prod_{i=1}^{M(N)}e^{NF_i}\right]
    \leq 
    I_N^{(1)}(\{D_l\},\{E_{lk}\})
    \leq 
    (1+\epsilon(N,\kappa))\mathbb E\left[\mathfrak 1_{\mathcal B_N(\kappa)}\prod_{i=1}^{M(N)}e^{NF_i}\right]
    \label{eq:p2_rankM_general_bounds},
\end{align}
where %analogously to \eqref{eq:P2_rank1_F_exponent_terms}:
\begin{align}
    F_i
    =
    \sum_{l=1}^L
    \underbrace{
    \theta_{l,i} \frac{\tilde{g}_{l,i}^T}{||\tilde g_{l,i}||_2} E_{ll}\frac{\tilde g_{l,i}}{||\tilde g_{l,i}||_2}
    }_{F_{ll,i}}
    +
    \sum_{k\neq l}
    \underbrace{
    \sqrt{\theta_{l,i} \theta_{k,i}} \frac{\tilde g_{l,i}^T}{||\tilde g_{l,i}||_2} E_{lk}\frac{\tilde g_{k,i}}{||\tilde g_{k,i}||_2}
    }_{F_{lk,i}}.
\end{align}
By applying the same reformulations as in the Proof of Theorem~\ref{thm:pL_rank1_integral}
\begin{align}
    NF_{ll,i}
    \geq
        N\theta_{l,i}v_{l,i} 
        - 
        |N-||\tilde{g}_{l,i}||_2^2|
        \theta_{l,i}(||E_{ll}||_\infty + |v_{l,i}|)
        +
        \theta_{l,i}(\tilde g_{l,i}^T E_{ll} \tilde g_{l,i} - v_{l,i} \tilde g_{l,i} ^T \tilde g_{l,i})
    \label{eq:exp_Q_ll_i_lower_bound},
\end{align}
and
\begin{align}
    NF_{lk,i}
    \geq
    \sqrt{\theta_{l,i} \theta_{k,i}}\tilde g_{l,i}^TE_{lk}\tilde g_{k,i}
    -
    \left|
        N-
        ||\tilde g_{l,i}||_2
        ||\tilde g_{k,i}||_2
    \right|
    \sqrt{\theta_{l,i}\theta_{k,i}}
    ||E_{lk}||_\infty
    \label{eq:Flki_lower_bound}.
\end{align}
Next, we replace the Schmidt orthogonalized Gaussian by the corresponding non-orthogonalized Gaussian. 
In the approach taken in \cite{guionnetFourierViewTransform2005} this is done by diagonalizing the $E$ matrix, corresponding to $E_{11}$ at $L=1$ in our formulation. 
This is, however, not viable in the coupled setting since $E_{lk}$ and $E_{ll}$ do not share an eigenspace.
Therefore, we have to repeat the derivation of the bounds. 
Following the formulation of Schmidt orthogonalization in \cite[Equation (24)]{guionnetFourierViewTransform2005} we attain the following bound on $|||\tilde g_{l,i}||_2^2-||g_{l,i}||_2^2|$ as
\begin{align}
    \Big|
        ||\tilde{g}_{l,i}||^2_2
        -
        ||g_{l,i}||_2^2
    \Big|
    &=
    \left|
        \left(g_{l,i} + \sum_{j=1}^{i-1}A_{ij}^{(l)} g_{l,j}\right)^T
        \left(g_{l,i} + \sum_{j=1}^{i-1}A_{ij}^{(l)} g_{l,j}\right)
        -g_{l,i}^Tg_{l,i}
    \right|
    \\
    &=
    \left|
        2g_{l,i}^T\sum_{j=1}^{i-1}A_{ij}^{(l)}g_{l,j}
        +
        \sum_{j=1}^{i-1}
            \sum_{\substack{j'=1\\j'\neq j}}^{i-1}
                A_{ij}^{(l)}
                A_{ij'}^{(l)}
                g_{l,j}^Tg_{l,j'}
        +\sum_{j=1}^{i-1}
            {A_{ij}^{(l)}}^2g_{l,j}^Tg_{l,j}
    \right|
    \\
    &\leq
        2\sum_{j=1}^{i-1}\left|A_{ij}^{(l)}\right|\left|g_{l,i}^Tg_{l,j}\right|
        +
        \sum_{j=1}^{i-1}
            \sum_{\substack{j'=1\\j'\neq j}}^{i-1}
                \left|A_{ij}^{(l)}\right|
                \left|A_{ij'}^{(l)}\right|
                \left|g_{l,j}^Tg_{l,j'}\right|
        +\sum_{j=1}^{i-1}
            \left|A_{ij}^{(l)}\right|^2\left|g_{l,j}^Tg_{l,j}\right|
    \label{eq:P2_rankM_bound_diff_orthogonalized},
\end{align}
where $A_{ij}^{(l)}$ are the orthogonalization coefficients for $\tilde g_{l,i}$.
In \cite[Section 2.2]{guionnetFourierViewTransform2005} it is established that, for $\mathcal B_N^\mathrm{GM}(\kappa)$, $\sup_{j<i}|A_{ij}|\leq c''N^{-\kappa}$.
The event $\mathcal B_N(\kappa)$ is similar to $\mathcal B_N^\mathrm{GM}(\kappa)$, and the bound on $A_{ij}$ holds, such that \eqref{eq:P2_rankM_bound_diff_orthogonalized} is bounded by
\begin{align}
    \left|
        ||\tilde{g}_{l,i}||^2_2
        -
        ||g_{l,i}||_2^2
    \right|
    \leq 2c''M(N)N^{1-2\kappa}
    +M(N)^2{c''}^2N^{1-3\kappa}
    +M(N){c''}^2N^{1-2\kappa}(N^{-\kappa}+1).
\end{align}
Now let $M(N) = \zeta N^{\rho}$ for some finite $\zeta>0$ as $N\uparrow\infty$ and $\rho<\kappa<1/2$, then for sufficiently large $N$ and some finite constant $c_3$
\begin{align}
    \left|
        ||\tilde{g}_{l,i}||^2_2
        -
        ||g_{l,i}||_2^2
    \right|
    \leq c_3N^{1-2\kappa+\rho}
    \label{eq:bound_orthogonalized_min_original}.
\end{align}
The bound \eqref{eq:exp_Q_ll_i_lower_bound} contains terms of the form $\tilde{g}_{l,i}^TE_{ll}\tilde g_{l,i}$. 
The corresponding error bound is given by
\begin{align}
    \label{eq:difference_bound_GS_orthogonalized_Fll_terms}
    \left|
        \tilde{g}_{l,i}^TE_{ll}\tilde g_{l,i}
        -
        g_{l,i}^TE_{ll}g_{l,i}
    \right|
    \leq 
    2c''N^{-\kappa} \sum_{j=1}^{i-1} \left|g_{l,i}^T E_{ll}g_{l,j}\right|
    +
    c''^2N^{-2\kappa} \sum_{j=1}^{i-1} \left|g_{l,j}^TE_{ll}g_{l,j}\right|
    +
    c''^2 N^{-2\kappa}\sum_{j=1}^{i-1}\sum_{\substack{j'=1\\j'\neq j}}^{i-1} \left|g_{l,j}^T E_{ll}g_{l,j'}\right|.
\end{align}
By definition, $E_{ll}$ is self-adjoint, and we can apply the following bound on the bilinear forms
\begin{align}
    |g_{l,i}^TE_{ll}g_{l,j} |
    \leq~&
    ||E_{ll}||_\infty ||g_{l,i}||_\infty||g_{l,j}||_\infty \leq ||E_{ll}||_\infty N(N^{-\kappa}+1)~ \forall ~i,j
    \label{eq:bound_bilinear_form_Ell}
\end{align}
Applying \eqref{eq:bound_bilinear_form_Ell} to \eqref{eq:difference_bound_GS_orthogonalized_Fll_terms} yields
\begin{align}
    \left|
        \tilde{g}_{l,i}^TE_{ll}\tilde g_{l,i}
        -
        g_{l,i}^TE_{ll}g_{l,i}
    \right|
    \leq 
    c_4 ||E_{ll}||_\infty N^{1-\kappa+\rho}
    \label{eq:bound_orthogonalized_vect_mat_vect}.
\end{align}
In order to bound $F_{lk,i}$ we first bound $\left|N-||\tilde g_{l,i}||_2||\tilde g_{k,i}||_2\right|$ by applying the Schmidt orthogonalization and computing upper and lower bounds of the expression in the absolute operation.
The bounds on $||\tilde g_{l,i}||_2$ are given as
\begin{align}
    &\sqrt{
        ||g_{l,i}||_2^2 
        - 
        2 \sum_{j=1}^{i-1} |A_{ij}^{(l)}||g_{l,i}^Tg_{l,j}|
        -
        \sum_{j'=1}^{i-1}\sum_{\substack{j''=1\\j''\neq j'}}^{i-1} |A_{ij'}^{(l)}||A_{ij''}^{(l)}||g_{l,j'}^Tg_{l,j''}|
    }
    \leq
    ||\tilde g_{l,i}||_2
    \\
    &\leq
    \nonumber
    \sqrt{    
        ||g_{l,i}||_2^2 
        + 
        2 \sum_{j=1}^{i-1} |A_{ij}^{(l)}||g_{l,i}^Tg_{l,j}|
        +
        \sum_{j'=1}^{i-1}\sum_{\substack{j''=1\\j''\neq j'}}^{i-1} |A_{ij'}^{(l)}||A_{ij''}^{(l)}||g_{l,j'}^Tg_{l,j''}|
    }.
\end{align}
With the bounds on $|A_{ij}^{(l)}|$ and $|g_{l,i}^Tg_{l,j}|$ under event $\mathcal B_N(\kappa)$ we obtain
\begin{align}
    &\sqrt{
        N(1-N^{-\kappa})
        - 
        2 c'' \zeta N^{1-2\kappa+\rho}
        -
        c''^2 \zeta^2 N^{1-3\kappa+2\rho}
        -
        c''^2 \zeta N^{1-2\kappa + \rho}(1+N^{-\kappa})
    }
    \leq
    ||\tilde g_{l,i}||_2
    \\
    &\leq
    \nonumber
    \sqrt{
        N(1+N^{-\kappa})
        + 
        2 c'' \zeta N^{1-2\kappa+\rho}
        +
        c''^2 \zeta^2 N^{1-3\kappa+2\rho}
        +
        c''^2 \zeta N^{1-2\kappa + \rho}(1+N^{-\kappa})
    }.
\end{align}
The roots yield real values for sufficiently large $N$ and $\rho<\kappa<1/2$.
The bounds are independent of $l$ and $i$, such that the upper and lower bounds for $||\tilde g_{l,i}||_2||\tilde g_{k,i}||_2$ are given by the square of the upper and lower bounds respectively.
Thus, we obtain the bound on $\left|N-||\tilde g_{l,i}||_2||\tilde g_{k,i}||_2\right|$ if $\rho<\kappa<1/2$ as
\begin{align}
    \left|N-||\tilde g_{l,i}||_2||\tilde g_{k,i}||_2\right|
    \leq c_5 N^{1-\kappa}
    \label{eq:lower_bound_N_min_length_orthogonalized_gaussians}.
\end{align}
The bound on $\left|\tilde g_{l,i}^TE_{lk}\tilde g_{k,i}-g_{l,i}^TE_{lk}g_{k,i}\right|~\forall~l\neq k$ is derived as
\begin{align}
     \Big|
        \tilde g_{l,i}^TE_{lk}\tilde g_{k,i}
        -
        g_{l,i}^TE_{lk}g_{k,i}
    \Big|
    &=
    \left|
    \left(
        g_{l,i}
        +\sum_{j=1}^{i-1}
            A_{ij}^{(l)}
            g_{l,j}
    \right)^T
    E_{lk}
    \left(
        g_{k,i}
        +\sum_{j=1}^{i-1}
            A_{ij}^{(k)}
            g_{k,j}
    \right)
    -
    g_{l,i}^TE_{lk}g_{k,i}
    \right|
    \\
    &=
    \left|
        \sum_{j=1}^{i-1}
        A_{ij}^{(k)}
        g_{l,i}^T
        E_{lk}
        g_{k,j}
        +
        \sum_{j=1}^{i-1}
        A_{ij}^{(l)}
        g_{l,j}^T
        E_{lk}
        g_{k,i}
        +
        \sum_{j',j''=1}^{i-1}
        A_{ij'}^{(l)}
        A_{ij''}^{(k)}
        g_{l,j'}^T
        E_{lk}
        g_{k,j''}
    \right|
    \\
    &\leq
    c''
    N^{-\kappa}
    \sum_{j=1}^{i-1}
    \left(
    ||g_{l,i}||_2||g_{k,j}||_2
    +
    ||g_{l,j}||_2||g_{k,i}||_2
    \right)
    ||E_{lk}||_\infty
    \label{eq:bound_bilinear_coupling_error}
    \\
    \nonumber
    &~~~+
    c''^2
    N^{-2\kappa}
    \sum_{j',j''=1}^{i-1}
    ||g_{l,j'}||_2
    ||g_{k,j''}||_2
    ||E_{lk}||_\infty.
\end{align}
Finally, the vector products in \eqref{eq:bound_bilinear_coupling_error} are bounded under the event $\mathcal B_N(\kappa)$ and $M(N)=\zeta N^{\rho}$, such that
\begin{align}
    \left|
        \tilde g_{l,i}^TE_{lk}\tilde g_{k,i}
        -
        g_{l,i}^TE_{lk}g_{k,i}
    \right|
    \leq~&
    ||E_{lk}||_\infty
    N^{1-\kappa}
    \left(
    2
    c''
    \zeta
    N^{\rho}
    +
    c''^2
    \zeta^2
    N^{-\kappa+2\rho}
    \right)
    \leq
    c_6N^{1-\kappa+\rho} ||E_{lk}||_\infty
    \label{eq:bound_orthogonalized_difference_lk}.
\end{align} 
By applying bounds \eqref{eq:bound_orthogonalized_min_original} and \eqref{eq:bound_orthogonalized_vect_mat_vect} to \eqref{eq:exp_Q_ll_i_lower_bound} we show that for some finite $c_7>0$ and large enough $N$
\begin{align}
    NF_{ll,i}
    \geq 
    N \theta_{l,i}v_{l,i} 
    +\theta_{l,i}\left(g_{l,i}^TE_{ll}g_{l,i}-v_{l,i}g_{l,i}^Tg_{l,i}\right)
    - c_7 |\theta_{l,i}|N^{1-\kappa+\rho} \left(||E_{ll}||_\infty+|v_{l,i}|\right).
\end{align}
The coupling terms $F_{lk,i}$ are bounded by applying bounds \eqref{eq:lower_bound_N_min_length_orthogonalized_gaussians} and \eqref{eq:bound_orthogonalized_difference_lk} to \eqref{eq:Flki_lower_bound}, which leads to
\begin{align}
    NF_{lk,i}
    \geq~&
    \sqrt{\theta_{l,i} \theta_{k,i}}g_{l,i}^TE_{lk}g_{k,i}
    -
    c_6\sqrt{\theta_{l,i}\theta_{k,i}}N^{1-\kappa+\rho}||E_{lk}||_\infty
    -c_5 \sqrt{\theta_{l,i}\theta_{k,i}}N^{1-\kappa}||E_{lk}||_\infty.
\end{align}
The lower bound of the $L$-fold coupled HCIZ integral is, therefore, given by
\begin{align}
    I_N^{(1)}&(\{D_l\},\{E_{lk}\})
    \geq~
    \exp\left\{
        - C_\mathrm {d} \sup_{l,i} |\theta_{l,i}|N^{1-\kappa+2\rho}\left(\sup_{l}||E_{ll}||_\infty+\sup_{l,i}|v_{l,i}|\right)
    \right\}
    \label{eq:P2_rankM_coupled_HCIZ_lower_bound}
    \\
    &\times 
    \exp\left\{
        - C_\mathrm c N^{1-\kappa +2\rho} \sup_{i,l\neq k}\sqrt{\theta_{l,i}\theta_{k,i}} ||E_{lk}||_\infty
    \right\}
    \prod_{i=1}^{M(N)}
    \prod_{l=1}^L
    \exp\left\{
        N \theta_{l,i} v_{l,i}
    \right\}
    \nonumber
    \\
    &\times
    \underbrace{
        \mathbb E\left[
            \mathfrak 1_{\mathcal B_N(\kappa)}
            \prod_{i=1}^{M(N)}
            \exp\left\{
                    \sum_{l=1}^L
                    \theta_{l,i}
                    \left(
                        g_{l,i}^TE_{ll}g_{l,i}
                        -
                        v_{l,i}g_{l,i}^Tg_{l,i}
                    \right)
                    +
                    \sum_{k\neq l}
                    \sqrt{\theta_{l,i}\theta_{k,i}}
                    g_{l,i}^TE_{lk}g_{k,i}
            \right\}
        \right]
    }_{\Xi_\mathrm{lb}^{M(N)}}
    \nonumber,
\end{align}
where $C_\mathrm{d}>0$ is a finite value characterizing the strength of the direct terms, and $C_\mathrm{c}>0$ is a finite value characterizing the strength of the coupled terms.
The influence of the coupling effect grows with $N^{1-\kappa+2\rho}$ and asymptotically vanishes upon taking the logarithm and normalizing with $N M(N)$, if $\rho<\kappa<1/2$.
Let $g = [g_{1,1}^T,\ldots g_{L,1}^T,\ldots,g_{L,M(N)}^T]^T \in \mathbb R^{LM(N)N\times 1}$ be a Gaussian vector of zero mean and identity covariance matrix. 
Then, with $g_i = [g_{1,i}^T,\ldots, g_{L,i}^T]^T$, we can write
\begin{align}
    \Xi_\mathrm{lb}^{M(N)}
    =~&
    \int 
        \mathfrak 1_{\mathcal B_N(\kappa)}
        \prod_{i=1}^{M(N)}
            \left|\Lambda_i\right|^{-\frac{1}{2}}
            \mathrm{d}\mu\left(g_i; 0, \Lambda_i^{-1}\right)      
    =
    P_{N}\left(\mathcal B_N(\kappa)\right)
    \prod_{i=1}^{M(N)}
        \left|\Lambda_i\right|^{-\frac{1}{2}}
    \label{eq:rankM_distribution_PN},
\end{align}
with $\Lambda_{i}$ defined analogously to $\Lambda$ in \eqref{eq:lambda_definition_rank_1}. 
For brevity, we define the two RVs $X_{l,i} = g_{l,i}^Tg_{l,i}/N-1$ and $Y_{l,ij}=g_{l,i}^Tg_{l,j}/N$.
Chebyshev's inequality is not tight enough to obtain a non-vanishing union bound on the probability of event $\mathcal B_N(\kappa)$ with respect to the probability measure $P_N(\cdot)$.
Instead, we apply sub-exponential tail bounds \cite[Chapter 2]{wainwrightHighDimensionalStatisticsNonAsymptotic2019} to the two types of converse sub-event probabilities:
\begin{enumerate}
    \item $P_N(\left|X_{l,i}\right|\geq N^{-\kappa})$ for any $l,i$.
    \item $P_N(\left|Y_{l,ij}\right|\geq N^{-\kappa})$ for any $l,i\neq j$.
\end{enumerate}
To facilitate the bound computation, $g_{l,i}^Tg_{l,j}/N$ must be of unit mean for $i=j$ and of zero mean otherwise. 
The latter is ensured for $i\neq j$ by \eqref{eq:rankM_distribution_PN}.
For $i=j$, the constraint on the mean is enforced by
\begin{align}
    \mathrm{Tr_L}\left[
        \mathbb E\left[
            g_ig_i^T
        \right]
    \right]
    =
    \mathfrak G_{\hat Z_{N,i}}^{\mathcal D_L}(\hat{S}_i)
    =
    \mathbb I_L,
    \label{eq:P2_rankM_OVStilTrans_condition}
\end{align}
which yields the same relation to the R-transform as in \eqref{eq:L_rank1_definition_Shat}.
The following Definition and Proposition are taken from \cite[Chapter 2]{wainwrightHighDimensionalStatisticsNonAsymptotic2019} and restated for the sake of completeness:
\begin{definition}
    An RV $X$ with mean $\mu=\mathbb E[X]$ is sub-exponential if there are non-negative parameters $(\nu, \chi)$ such that
    \begin{equation*}
        \mathbb E\left[
            \exp \left\{
                \gamma (X-\mu)
            \right\}
        \right]
        \leq
        \mathrm{e}^{\frac{\nu^2\gamma^2}{2}} 
        ~\forall |\gamma|<\frac{1}{\chi}
    \end{equation*}
    \label{def:sub-exponentialRV}
\end{definition}
If $X$ is sub-exponential in the sense of Definition~\ref{def:sub-exponentialRV}, then an exponential tail bound can be applied.
\begin{proposition}
    Suppose that $X$ is sub-exponential with parameters $(\nu,\chi)$. Then
    \begin{equation*}
        \mathrm{Pr}\left\{
            \left|X-\mu\right| \geq t
        \right\}
        \leq
        \begin{cases}
            2\exp\left\{
                -\frac{t^2}{2\nu^2}
            \right\}
            & \text{if $0\leq t \leq \frac{\nu^2}{\chi}$}
            \vspace{2pt}
            \\
            2\exp\left\{
                -\frac{t}{2\chi}
            \right\}
            & \text{if $t>\frac{\nu^2}{\chi}$}
        \end{cases}
    \end{equation*}
    \label{proposition:sub-exponential_tail_bounds}
\end{proposition}
For brevity we introduce the shorthand notation
\begin{equation}
    C_{lk,i}
    =
    \left[\Lambda_i^{-1}\right]_{lk}
    \in \mathbb R^{N\times N}
    ,
\end{equation}
with eigenvalues $\psi_{li,n}~\forall n\in[N]$ if $l=k$.
The moment generating function for $X_{l,i}$ is computed by Gaussian integration, as
\begin{align}
    \mathbb E\left[
        \exp\left\{
            \gamma X_{l,i}
        \right\}
    \right]
    =
    \exp\left\{
        -\gamma 
        -\frac{1}{2}\log \left|
            \mathbb I -\frac{2\gamma}{N}C_{ll,i}
        \right|
    \right\}
    .
    \label{eq:MGD_X_li}
\end{align}
In order to find the corresponding $(\nu,\chi)$ parameters to $X_{l,i}$ we bound \eqref{eq:MGD_X_li} by some $\exp\{\nu^2\gamma^2/2\}$ for all $|\gamma|<1/\chi$.
This is equivalent to the condition
\begin{align}
    f_{X_{l,i}}(\gamma) = \frac{\nu^2\gamma^2}{2} + \gamma + \frac{1}{2}\sum_{n=1}^N\log \left(
            1 -\frac{2\gamma}{N}\psi_{li,n}
        \right)
        \geq 0, & \text{ for all } {|\gamma|<\frac{1}{\chi}}.
\end{align}
Note that $f_{X_{l,i}}(0)=0$, such that we can check for a negative (positive) derivative for $\gamma<0$ ($\gamma>0$).
The derivative is given as
\begin{align}
    f_{X_{l,i}}'(\gamma) =  
    \frac{1}{N}\sum_{n=1}^N
        \frac{
            \left(
                \frac{2\nu^2\gamma}{N}\psi_{li,n}
                -
                \nu^2
                +
                \frac{2}{N}\psi_{li,n}^2
            \right)
            \gamma
        }{
            \frac{2\gamma}{N}\psi_{li,n}-1
        },
    \label{eq:derivfFuncX}
\end{align}
with $1=\frac{1}{N}\sum_{n=1}^N \psi_{li,n}$ by \eqref{eq:P2_rankM_OVStilTrans_condition}.
Note that $\Lambda_i^{-1}\succ 0$ and consider $\gamma<0$, i.e. all denominator in \eqref{eq:derivfFuncX} are negative.
To obtain a negative derivative for all $\gamma<0$, we must have $2\psi_{li,n}^2/N-\nu^2\leq 0$.
This is enforced for all $n$ by 
\begin{equation}
    \nu^2 \geq \frac{2}{N}\left|\left|
        C_{ll,i}
    \right|\right|_\infty^2
    =\nu^2_\mathrm{min}.
\end{equation}
In the case of $\gamma>0$, all denominators are negative when 
\begin{equation}
\gamma\leq\frac{N}{2\left|\left|
        C_{ll,i}
    \right|\right|_\infty}.
\end{equation}
Therefore, the numerator must be negative as well, i.e. 
\begin{align}
    \gamma
    \leq
        \frac{N}{2 \psi_{li,n}} - \frac{\psi_{li,n}}{\nu^2},
\end{align}
which is guaranteed by 
\begin{align}
    \gamma
    \leq
    \frac{N}{4 \left|\left|
        C_{ll,i}
    \right|\right|_\infty},
\end{align}
with $\nu^2\geq 2 \nu_\mathrm{min}^2$.
We conclude that $X_{l,i}$ is sub-exponential with parameters
\begin{equation}
    \left(
        \frac{4}{N}
        \left|\left|
            C_{ll,i}
        \right|\right|_\infty^2
        ,
        \frac{
            4
            \left|\left|
                C_{ll,i}
            \right|\right|_\infty
        }{N}
    \right).
\end{equation}
The moment generating function corresponding to the sub-event probability of $Y_{l,ij}$ is bounded by
\begin{align}
     \mathbb E\left[
        \exp\left\{
            \gamma Y_{l,ij}
        \right\}
    \right]
    =~&
    \exp\left\{
        -\frac{1}{2}\log \left|
            \mathbb I -\frac{\gamma^2}{N^2}C_{ll,i}C_{ll,j}
        \right|
    \right\}
    \label{eq:mgf_Y_uncorrelated}
    \leq
    \exp\left\{
        -\frac{1}{2}N\log \left(
            1 
            -\frac{\gamma^2}{N^2}
            \left|\left|C_{ll,i}\right|\right|_\infty 
            \left|\left|C_{ll,j}\right|\right|_\infty
        \right)
    \right\},
\end{align}
for
\begin{align}
    \gamma^2 <  \frac{N^2}{||C_{ll,i}||_\infty ||C_{ll,j||_\infty}} = \gamma_\mathrm{max}^2,
\end{align}
i.e. the argument of the log-determinant is positive definite.
Considering the logarithmic inequality
\begin{equation}
    -\log(1-x)\leq 2x ~\forall~ 0\leq x \leq \frac{1}{2},
\end{equation}
it is straightforward to see, that $Y_{l,ij}$ is sub-exponential with parameters
\begin{equation}
    \left(
        \frac{2}{N}
        \left|\left|C_{ll,i}\right|\right|_\infty 
        \left|\left|C_{ll,j}\right|\right|_\infty
        ,
        \frac{\sqrt{
            2
            \left|\left|C_{ll,i}\right|\right|_\infty 
            \left|\left|C_{ll,j}\right|\right|_\infty 
        }}{N}
    \right).
\end{equation}
The tail bounds of Proposition~\ref{proposition:sub-exponential_tail_bounds} can hence be applied for both types of converse sub-event probabilities.
Furthermore, with $t=N^{-\kappa}$ we have $0\leq t \leq \nu^2/\chi$ for sufficiently large $N$. 
Bound \eqref{eq:P2_rank1_eta_bound} is still valid in the case of $M(N)>1$ such that $||C_{ll,i}||_\infty\leq (2\eta \sup_{l,i} \theta_{l,i})^{-1}$ for some $\eta>0$. 
Applying the union bound to the event probability of the converse $\mathcal B_N^c(\kappa)$ yields
\begin{align}
    &P_N(\mathcal B_N(\kappa)) 
    \geq 
    1
    - 
    2LM(N)\exp\left\{
        -\frac{1}{2}N^{1-2\kappa}\eta^2 \sup_{l,i} \theta_{l,i}^2
    \right\}
    - 
    2L{M(N) \choose 2} \exp\left\{
        -\frac{1}{4}N^{1-2\kappa}\eta^2 \sup_{l,i} \theta_{l,i}^2
    \right\}.
\end{align}
For $\kappa<1/2$, the exponentials vanish faster than the scaling terms grow, such that we have $P_N(\mathcal B_N(\kappa))\geq 1/2$ for sufficiently large $N$.
The lower bound on the $L$-fold coupled spherical integral of rank $M(N)=O\left(N^{1/2-\epsilon}\right)$ for any $\epsilon>0$ is hence given by
\begin{align}
    I_N^{(1)}(\{D_l\},\{E_{lk}\})
    \geq~&
    \frac{1}{2}
    \exp\left\{
        - C_\mathrm {d} N^{1-\kappa+2\rho} \sup_{l,i} |\theta_{l,i}|\left(||E_{ll}||_\infty+|v_{l,i}|\right)
    \right\}
    \label{eq:lower_bound_rankMN}
    \exp\left\{
        - C_\mathrm c N^{1-\kappa +2\rho} \sup_{i,l,k}\sqrt{\theta_{l,i}\theta_{k,i}} ||E_{lk}||_\infty
    \right\}
    \\
    &\times
    \prod_{i=1}^{\zeta N^\rho}
        \left|\Lambda_i\right|^{-\frac{1}{2}}
        \prod_{l=1}^L
            \exp\left\{
                N \theta_{l,i} v_{l,i}
            \right\}
    \nonumber.
\end{align}
As in the proof of Theorem \ref{thm:pL_rank1_integral}, the upper bound is obtained similarly, by inverting the signs of the error terms in the exponent and $P_N(\mathcal B_N(\kappa))\leq 1$.
Hence, by applying the normalized logarithm to the coupled integral as well as the corresponding bounds we obtain
\begin{align}
    &
    \frac{1}{\zeta N^{1+\rho}}\log\left(\frac{1}{2}\right)
    - C_\mathrm {d}\frac{N^{1-\kappa+2\rho}}{\zeta N^{1+\rho}} \sup_{l,i} |\theta_{l,i}|\left(||E_{ll}||_\infty+|v_{l,i}|\right)
    - C_\mathrm c \frac{N^{1-\kappa +2\rho}}{\zeta N^{1+\rho}} \sup_{i,l,k}\sqrt{\theta_{l,i}\theta_{k,i}} ||E_{lk}||_\infty
    \\
    \nonumber
    &\leq \frac{1}{NM(N)} \log I_N^{(1)}(\{D_l\},\{E_{lk}\})
    -
    \frac{1}{M(N)}
    \sum_{i=1}^{M(N)}
        \theta_{l,i}v_{l,i}
        - \frac{1}{2N}\log \left|\Lambda_i\right|
    \\
    \nonumber
    &\leq
    \frac{1}{NM(N)}\log\left(1+\epsilon(N,\kappa)\right)
    + C_\mathrm {d}\frac{N^{1-\kappa+2\rho}}{\zeta N^{1+\rho}} \sup_{l,i} |\theta_{l,i}|\left(||E_{ll}||_\infty+|v_{l,i}|\right)
    + C_\mathrm c \frac{N^{1-\kappa +2\rho}}{\zeta N^{1+\rho}} \sup_{i,l,k}\sqrt{\theta_{l,i}\theta_{k,i}} ||E_{lk}||_\infty.
\end{align}
If $\rho<\kappa<1/2$, the upper and lower bounds tend to zero as $N\uparrow\infty$.
Assuming the joint empirical spectral measure $\hat \mu_{N}(\theta_{1},\ldots,\theta_{L})$ converges to the joint measure $\mu(\theta_1,\ldots,\theta_{L})$ as $N\uparrow\infty$ we write
\begin{align}
    \lim_{N\uparrow\infty}
    \frac{1}{NM(N)} \log I_N^{(1)}(\{D_l\},\{E_{lk}\})
    =
    \int
    g(\theta_{1},\ldots, \theta_{L})
    \mathrm d \mu(\theta_{1},\ldots,\theta_{L}),
\end{align}
with $g:\mathbb R^L\to\mathbb R$ as defined in \eqref{eq:g_func_definition_rank_1}.
The proof is extended to $\beta=2$ by the same procedure as described in the proof of Theorem \ref{thm:pL_rank1_integral}.
\end{proof}

\begin{proof}[Proof of Theorem~\upshape{\ref{thm:pL_rankM_integral_modification}}]
    Replace $E_{lk}$ by $E_{lk,i}$ in \eqref{eq:integral_P2_rankM}.
    Then, all derivations up to \eqref{eq:P2_rankM_coupled_HCIZ_lower_bound} still hold, where $\sup_l$ must be replaced by $\sup_{l,i}$.
    Note that in $\Lambda_i^{-1}$, not only the $\theta_{l,i}$ are dependent on index $i$, but so do the matrices $E_{lk,i}$.
    We, hence, observe that matrix $Z_{N,i}$ is also dependent on index $i$, such that we obtain
    \begin{align}
        \mathrm{Tr}_L\left[
            g_ig_i^T
        \right]
        =
        \mathrm{Tr}_L\left[
            \left(
                -P_i^{\frac 1 2}
                Z_{N,i} 
                P_i^{\frac 1 2}
                - S_i
                \otimes \mathbb I_N
            \right)^{-1}
        \right]
        = 
        -
        P_i^{-\frac 1 2}
        \mathfrak G_{Z_{N,i}}^{\mathcal D_L}(-S_i)
        P_i^{-\frac 1 2}
        =\mathbb I_L,
    \end{align}
    with $Z_{N,i}\in\mathbb A^{LN\times LN}$ and $\left[Z_{N,i}\right]_{lk}=E_{lk,i}$.
    All subsequent bounds of the Proof of Theorem~\ref{thm:pL_rankM_integral} still hold, such that we obtain
    \begin{align}
        \lim_{N\uparrow\infty} 
        \frac 1 {NM(N)} 
        \log J_N^{(1)}\left(
            \{D_l\}, \{E_{lk,i}\}
        \right)
        =
        \frac{1}{M(N)}
        \sum_{i=1}^{M(N)}
            g_i\left(\theta_{1,i},\ldots, \theta_{L,i}\right).
    \end{align}
    In this case however, function $g$ is conditioned on index $i$ as $\Lambda_i$ is not only index dependent through the singular values $\theta_{l,i}$ but also $Z_{N,i}$.
    Reformulation of the function through derivation and subsequent integration are performed as for Theorem \ref{thm:pL_rank1_integral}.
    This concludes the proof for $\beta=1$. 
    For $\beta=2$, the same steps as in the Proof of Theorem \ref{thm:pL_rank1_integral} are followed.
\end{proof}

%%%%%%%%%%%%%%%%%%%%%%%%%%%%%%%%%%%%%%%%%%%%%%
%% Single Appendix:                         %%
%%%%%%%%%%%%%%%%%%%%%%%%%%%%%%%%%%%%%%%%%%%%%%
%\begin{appendix}
%\section*{???}%% if no title is needed, leave empty \section*{}.
%\end{appendix}
%%%%%%%%%%%%%%%%%%%%%%%%%%%%%%%%%%%%%%%%%%%%%%
%% Multiple Appendixes:                     %%
%%%%%%%%%%%%%%%%%%%%%%%%%%%%%%%%%%%%%%%%%%%%%%
%\begin{appendix}
%\section{???}
%
%\section{???}
%
%\end{appendix}

%%%%%%%%%%%%%%%%%%%%%%%%%%%%%%%%%%%%%%%%%%%%%%
%% Support information, if any,             %%
%% should be provided in the                %%
%% Acknowledgements section.                %%
%%%%%%%%%%%%%%%%%%%%%%%%%%%%%%%%%%%%%%%%%%%%%%
\begin{acks}[Acknowledgments]
The authors would like to thank Roland Speicher for discussions on the operator-valued R-transform and enabling a visit to the Department of Mathematics at Saarland University.
Furthermore, we thank Hermann Schulz-Baldes for editorial help and guidance on the publication process in the mathematical world, and Johanna S. Fröhlich for proofreading the manuscript.
\end{acks}
%%%%%%%%%%%%%%%%%%%%%%%%%%%%%%%%%%%%%%%%%%%%%%
%% Funding information, if any,             %%
%% should be provided in the                %%
%% funding section.                         %%
%%%%%%%%%%%%%%%%%%%%%%%%%%%%%%%%%%%%%%%%%%%%%%
\begin{funding}
The authors were supported by Deutsche Forschungsgemeinschaft (DFG) under the project Computation Coding (MU-3735/8-1).
\end{funding}

%%%%%%%%%%%%%%%%%%%%%%%%%%%%%%%%%%%%%%%%%%%%%%%%%%%%%%%%%%%%%
%%                  The Bibliography                       %%
%%                                                         %%
%%  imsart-number.bst  will be used to                     %%
%%  create a .BBL file for submission.                     %%
%%                                                         %%
%%  Note that the displayed Bibliography will not          %%
%%  necessarily be rendered by Latex exactly as specified  %%
%%  in the online Instructions for Authors.                %%
%%                                                         %%
%%  MR numbers will be added by VTeX.                      %%
%%                                                         %%
%%  Use \cite{...} to cite references in text.             %%
%%                                                         %%
%%%%%%%%%%%%%%%%%%%%%%%%%%%%%%%%%%%%%%%%%%%%%%%%%%%%%%%%%%%%%

%% if your bibliography is in bibtex format, uncomment commands:
\bibliographystyle{imsart-number} % Style BST file
\bibliography{zoteroBetterBibTeX}       % Bibliography file (usually '*.bib')

%% or include bibliography directly:
% \begin{thebibliography}{}
% \bibitem{b1}
% \end{thebibliography}

\end{document}